\UseRawInputEncoding 

\RequirePackage{fix-cm}
\documentclass{svjour3}                     
\smartqed  
\usepackage{amssymb}
\usepackage{graphicx}
\usepackage{xcolor}
\usepackage{graphicx}
\usepackage{amsmath}
\usepackage{epstopdf} 
\usepackage{mathptmx}      

\begin{document}

\title{A spectral approach to non-linear weakly singular fractional integro-differential equations}
\author{Amin Faghih$^{1}$ \and Magda Rebelo$^{2^*}$      
}
\institute{$1$ Contributing author (Amin Faghih)\\
Department of Applied Mathematics, Faculty of Basic Sciences, Sahand University of Technology, Tabriz, Iran
              \email{amin.fagheh71@gmail.com}\\
$2^*$ Corresponding Author (Magda Rebelo)\\
Center for Mathematics and Applications (NovaMath) and Department of Mathematics, FCT NOVA, Quinta da Torre, 2829-516, Caparica, Portugal
              \email{msjr@fct.unl.pt}
             }

\maketitle
\begin{abstract}
 In this work, a class of non-linear weakly singular fractional integro-differential equations is considered, and we first prove existence, uniqueness, and smoothness properties of the solution under certain assumptions on the given data. We propose a numerical method based on spectral Petrov-Galerkin method that handling to the  non-smooth behavior of the solution.
  The most outstanding feature of our approach is to evaluate the approximate solution by means of recurrence relations despite solving complex non-linear algebraic system. Furthermore, the well-known exponential accuracy is established in $L^{2}$-norm, and we provide some  examples to illustrate the theoretical results and the performance of the proposed method.
\end{abstract}
\section{Introduction}
The subject of fractional calculus has recently gained significant popularity and importance, due mainly to its memory features and demonstrated applications in numerous seemingly diverse and widespread fields of science and engineering. For further details, readers are relegated to the books (\cite{fo2}, \cite{39} ,\cite{30},\cite{fo1}, \cite{31}, \cite{r2-2}) and review papers (\cite{ref1}, \cite{ref2}, \cite{ref3}).

There is no generally applicable method to find an analytic solution to an arbitrary given fractional-integro differential equation (FIDE). That is why effective numerical methods can help overcome the problems caused by the shortage of analytical methods for the computation of solutions to FIDEs. Various kinds of approximate methods have independently appeared for the numerical solution of FIDEs along with smooth kernel function, such as Quadratic method \cite{1}, spline collocation method \cite{3,spl}, differential transform method \cite{8}, Legendre wavelet method \cite{9}, second Chebyshev wavelet method \cite{sc}, Laguerre collocation method \cite{11}, Jacobi collocation method \cite{14}, Taylor expansion method \cite{15}, Legendre collocation method \cite{16} and hybrid collocation method \cite{17}.

Weakly singular FIDEs seem to be investigated less frequently than FIDEs associated with smooth kernel function. For instance, alternative Legendre polynomials method \cite{10}, Jacobi Tau method \cite{24}, second kind Chebyshev spectral method \cite{21}, second kind Chebyshev wavelet method \cite{23}, and piecewise polynomial collocation method \cite{22} have introduced for the numerical solution of linear weakly singular FIDEs. Moreover, the most frequently used methods for obtaining the approximate solutions of non-linear weakly singular FIDEs are Legendre wavelet method \cite{20}, shifted Jacobi collocation method \cite{19}, and hat functions method \cite{18}. The vast majority of the aforementioned schemes lack The following critical properties playing an active role in constituting an effective numerical approach:
\begin{itemize}
  \item[$\bullet$] Absence of comprehensive analysis, including existence, uniqueness of the solution, and a vigorous smoothness survey which is crucial in establishing exponentially accurate spectral methods.
  \item[$\bullet$] Obtaining approximations through solving complicated non-linear algebraic systems along with high computational costs.
   Such snag drives numerical schemes toward generating low accurate approximate solutions to the equations arising from real-world phenomena, which are mainly regarded on the long domain.
  \item[$\bullet$] Lack of consistency between non-smooth behavior of the exact solution and basis functions that addresses researchers to deal with infinitely smooth basis functions such as polynomials.
      Indeed, generating a consistency between the degree of smoothness of the exact solution and the asymptotic behavior of the basis functions makes it possible to create high-order methods.
\end{itemize}

‎In this paper‎, the goal is to present a comprehensive investigation taking the policy of tackling the above difficulties to the following non-linear weakly singular FIDE
‎\begin{equation}\label{eq1}‎
‎\begin{cases}‎
‎D^{\alpha}_C y(t)=f(t,y(t))+\lambda \int_{0}^{t}(t-s)^{{\beta}-1}g(t,s,y(s))ds,\quad ‎t \in \Lambda=[0,T],\\‎
y^{(k)}(0)=y_{0}^{(k)},~k=0,1,\ldots‎, ‎\lceil \alpha \rceil-1,‎ ~ \quad \quad \quad \quad \alpha >0, \quad 0 < \beta \leq 1,
‎\end{cases}‎
‎\end{equation}‎
where $\lambda \in \mathbb{R}$, $\alpha=\frac{a_{1}}{b_{1}}$, $\beta=\frac{a_{2}}{b_{2}}$ along with $a_{i} \geq 1$, $b_{i} \geq2$, $i=1,2$, and $(a_{i},b_{i})$ are the pairs of co-prime integers.\\
Here $\lceil.\rceil$ is the ceiling function, and $T$ is a finite positive real value. $f(t, y(t)): \Lambda \times \mathbb{R} \rightarrow \mathbb{R}$ and $g(t, s, y(s)): D \times  \mathbb{R} \rightarrow \mathbb{R}$, with $\displaystyle D=\left\{(t,s) \in \Lambda\times \Lambda:\, 0\leq s\leq t\leq T\right\}$ are continuous functions, and $y(t): \Lambda \rightarrow \mathbb{R}$ is the unknown. $D_{C}^{\alpha}$ is Caputo fractional derivative of order $\alpha$ defined by
‎\begin{equation*}‎
‎D^{\alpha}_C (.)=I^{\lceil \alpha \rceil-\alpha}\partial_t^{\lceil \alpha \rceil}(.)‎,
‎\end{equation*}‎
‎in which $I^{\lceil \alpha \rceil-\alpha}$ denotes the Riemann-Liouville fractional integral operator of order $\lceil \alpha \rceil-\alpha$‎
 ‎\cite{39,30,31}‎.

In this respect, a comprehensive investigation regarding the existence, uniqueness, and smoothness properties of \eqref{eq1} are provided on the one side, and then an efficient spectral scheme is implemented to \eqref{eq1} thanks to the fractional generalized Jacobi functions (FGJFs) introduced in \cite{sevom}. Needless to direct that spectral methods offer highly accurate approximations for smooth problems. However, some cons are still existed, including the requirement of solving ambiguous and ill-conditioned algebraic systems and the striking decline in the accuracy of the approximations facing the problems with non-smooth solutions. Contrary to all cons, the numerical strategy in this essay is taken whereby it contributes to both spectral accuracy in attacking non-smooth solution and evaluating approximate solution by means of recurrence relations despite solving non-linear complex algebraic systems.

This paper is arranged in the following way. We begin by presenting theorems of existence, uniqueness, and smoothness. This analysis confirms that some derivatives of the solution have probably a discontinuity at the initial point. Subsequently, in Section \ref{sec2}, we first provide the essential concepts and definitions of generalized Jacobi polynomials (GJPs) and FGJFs, and then a state-of-the-art Petrov-Galerkin method is implemented to deal with \eqref{eq1}. Numerical solvability and practical implementation of the relevant non-linear algebraic system are examined as well. In particular, the error estimate is mightily surveyed in Section \ref{sec3}. Section \ref{sec4} includes some prototype examples to assess the efficiency and applicability of the introduced method. Section \ref{sec5} ultimately presents concluding remarks.
\section{Existence, uniqueness and smoothness}\label{sec1}
We allocate this section to the existence, uniqueness, and smoothness properties of the solution of \eqref{eq1}. We first provide the following theorem regarding the existence and uniqueness of the solution.

Let $\displaystyle{\psi(t)=\sum_{k=0}^{\lceil\alpha\rceil-1} \frac{t^{k}}{k !} y_{0}^{(k)}} $. Define the set
$$\Omega_{\zeta}=\left\{ y\in C([0,T]):\,  \|y-\psi\|_{\Lambda}\leq \zeta   \right\},$$ where
$\displaystyle \|.\|_{\Lambda}=\max_{t\in \Lambda=[0,T]}|z(t)|$ for all $z\in C(\Lambda)$ and
\begin{equation}\label{defgamma}
\zeta=
\frac{\|f\|_{\Lambda} T^{\alpha}}{\Gamma(\alpha+1)}+\frac{\|g\|_{\Lambda}\lambda \Gamma(\beta)T^{\alpha+\beta}}{\Gamma(1+\alpha+\beta)}.
\end{equation}

Through applying the Riemann-Liouville fractional integral operator of order $\alpha$, \eqref{eq1} is changed into the weakly singular Volterra integral equation
\begin{equation}\label{r3}
y(t)=\psi(t)+I^{\alpha} f(t, y(t))+\lambda I^{\alpha}\left(\int_{0}^{t}(t-s)^{\beta-1} g(t, s, y(s)) d s\right).
\end{equation}

At this stage, we define an operator $\mathcal{T}_{\psi}$, on $\Omega_{\zeta}$, by
\begin{eqnarray*}
\nonumber \mathcal{T}_{\psi}\left( \varphi\right)(t)&=&\psi(t)+I^{\alpha} f(t, \varphi(t))+\lambda I^{\alpha}\left(\int_{0}^{t}(t-s)^{\beta-1} g(t, s, \varphi(s)) d s\right)\\ \nonumber
&=& \psi(t)+
\frac{1}{\Gamma(\alpha)}\int_0^t(t-s)^{\alpha-1}f(s,\varphi(s))
ds \\
\label{defop} &+&\frac{\lambda}{\Gamma(\alpha)} \int_0^t (t-\tau)^{\alpha-1}\left(\int_{0}^{\tau}(\tau-s)^{\beta-1} g(\tau, s, \varphi(s)) d s\right).
\end{eqnarray*}
Using this operator, the equation \eqref{eq1} can be rewritten as $\displaystyle y=\mathcal{T}_{\psi}\left( y\right)$. Thereby, if the operator $\mathcal{T}_{\psi}$ has a unique fixed point on $\Omega_{\zeta}$, \eqref{eq1} will possess a unique continuous solution.

\vspace*{0.5cm}
\begin{theorem}[Existence and uniqueness]\label{thm3}
Let $\displaystyle D_1= \Lambda \times [m_{\psi}-\zeta,M_{\psi}+\zeta]$ and $\displaystyle D_2= D \times [m_{\psi}-\zeta,M_{\psi}+\zeta],$ where $\displaystyle m_{\psi}=\min_{t\in \Lambda}|\psi(t)|$, $\displaystyle  M_{\psi}=\max_{t\in \Lambda}|\psi(t)|$ and $\zeta$ is defined by \eqref{defgamma}.

Assume that the functions  $f(s, y(s)): D_1  \rightarrow \mathbb{R}$  and $g(t, s, y(s)): D_2  \rightarrow \mathbb{R}$ are continuous for all $s\in \Lambda$ and further assume that the functions $f$ and $g$ fulfills Lipschitz conditions with respect to the second and third variables, respectively. Therefore, there exists $L_1>0$ and $L_2>0$ such that
\begin{eqnarray}
\nonumber
&& \lvert f\left(t,z_1\right)-f\left(t,z_2\right) \rvert  \leq L_{1}\lvert z_1-z_2\rvert,\quad \quad ~  \forall z_1,z_2 \in \Omega_{\zeta},\\
\nonumber
&& \lvert g\left(t,s,z_1\right)-g\left(t,s,z_2\right) \rvert \leq L_{2}\lvert z_1-z_2\rvert ,\quad \forall z_1,z_2 \in \Omega_{\zeta}.
\end{eqnarray}
Then, \eqref{eq1} owns a unique continuous solution on $\Lambda$.
\end{theorem}
\begin{proof}
Suppose that $\displaystyle y \in \Omega_{\zeta}.$ It is straightforward prove that $\displaystyle  \mathcal{T}_{\psi}\left( y \right)\in \Omega_{\zeta}$.

Let $\delta>0$ be a constant such that
\begin{equation}\label{defdelta}
\max \left(\frac{L_{1}}{\delta^{\alpha}}+\frac{L_{2} \lambda \Gamma(\beta)}{\delta^{\alpha+\beta}}\right)<1 .
\end{equation}

We introduce a new norm $\|\cdot\|_{\delta}$ over the space $C(\Lambda ; \Omega_{\zeta})$ as
$$
\|\varphi\|_{\delta}=\left\|\displaystyle \frac{\varphi}{\exp (\delta t)}\right\|_{\Lambda}.
$$

Using standard arguments, it can be readily inferred that $\Omega_{\zeta}$ is a closed subset of the Banach space of continuous functions
on $\Lambda$, associated with the norm $\|\cdot\|_{\delta}$.

 Let $\varphi, \hat{\varphi} \in \Omega_{\zeta}$. Regarding the Lipschitz assumption on $f$ and $g$, it follows
 \begin{eqnarray}\nonumber
 &&\frac{\left|\mathcal{T}_{\psi} \varphi(t)-\mathcal{T}_{\psi} \hat{\varphi}(t)\right|}{\exp (\delta t)} \leq \frac{L_{1}}{\Gamma(\alpha) \exp (\delta t)} \int_{0}^{t}(t-s)^{\alpha-1} \exp (\delta s) \frac{| \varphi(s)-\hat{\varphi}(s)|}{\exp (\delta s)} d s \\
 \nonumber
 &+& \frac{L_{2} \lambda}{\Gamma(\alpha) \exp (\delta t)} \int_{0}^{t}(t-s)^{\alpha-1}\left(\int_{0}^{s}(s-\tau)^{\beta-1} \exp (\delta \tau) \frac{|\varphi(\tau)-\hat{\varphi}(\tau)|}{\exp (\delta \tau)} d \tau\right) d s \\
\nonumber &\leq& \frac{L_{1}}{\Gamma(\alpha) \exp (\delta t)}\| \varphi-\hat{\varphi}\|_{\delta}
 \int_{0}^{t}(t-s)^{\alpha-1} \exp (\delta s) d s \\
\nonumber  &+&\frac{L_{2} \lambda}{\Gamma(\alpha) \exp (\delta t)} \|\varphi-\hat{\varphi}\|_{\delta} \int_{0}^{t}(t-s)^{\alpha-1}\left(\int_{0}^{s}(s-\tau)^{\beta-1} \exp (\delta \tau) d \tau\right) d s \\
\nonumber  &\leq& \frac{L_{1}}{\Gamma(\alpha) \delta^{\alpha}}\|\varphi-\hat{\varphi}\|_{\delta} \int_{0}^{\delta t} u^{\alpha-1} \exp (-u) d u \\
\nonumber  &+& \frac{L_{2} \lambda}{\Gamma(\alpha) \exp (\delta t) \delta^{\beta}}\|\varphi-\hat{\varphi}\|_{\delta} \int_{0}^{t}(t-s)^{\alpha-1} \exp (\delta s)\left(\int_{0}^{\delta s} v^{\beta-1} \exp (-v) d v\right) d s .
\end{eqnarray}
Due to the definition of Gamma function, we have
\begin{eqnarray}\nonumber
\frac{\left|\mathcal{T}_{\psi} \varphi(t)-\mathcal{T}_{\psi} \hat{\varphi}(t)\right|}{\exp (\delta t)} &\leq& \frac{L_{1}}{\delta^{\alpha}}\|\varphi-\hat{\varphi}\|_{\delta}+ \frac{L_{2} \lambda \Gamma(\beta)}{\exp(\delta t)\delta^{\beta} \Gamma(\alpha)} \int_0^t(t-s)^{\alpha-1}\exp(\delta s)ds\\
\nonumber &\leq &
\frac{L_{1}}{\delta^{\alpha}}\|\varphi-\hat{\varphi}\|_{\delta}+ \frac{L_{2} \lambda \Gamma(\beta)}{\delta^{\alpha +\beta} } \|\varphi-\hat{\varphi}\|_{\delta}
\\
\nonumber &=&\left(\frac{L_{1}}{\delta^{\alpha}}+\frac{L_{2} \lambda \Gamma(\beta)}{\delta^{\alpha +\beta} }\right)\|\varphi-\hat{\varphi}\|_{\delta}.\end{eqnarray}

Then from the above inequality and \eqref{defdelta}, it follows
\begin{equation}\nonumber
\left\|\mathcal{T}_{\psi} \left(\varphi\right)-\mathcal{T}_{\psi} \left(\hat{\varphi}\right)\right\|_{\delta}\leq
\left(\frac{L_{1}}{\delta^{\alpha}}+\frac{L_{2} \lambda \Gamma(\beta)}{\delta^{\alpha +\beta} }\right)\|\varphi-\hat{\varphi}\|_{\delta}< \|\varphi-\hat{\varphi}\|_{\delta}.\end{equation}

Therefore  the operator $\mathcal{T}_{\psi}$ is a contraction on $\Omega_{\zeta}$. Finally, by the Banach fixed
point principle, the proof of the theorem is
complete.
\end{proof}
We present the following theorem to dig into the asymptotic behavior of the solution of \eqref{eq1} through its series representation near the origin.
‎\begin{theorem}[Smoothness]\label{th202} Let the given continues functions $f(t,y(t))$ and $g(t,s,y(s))$ can be written as‎
‎\begin{eqnarray*}‎
f(t,y(t))&=&\bar{f}(t^{1/b},y(t))‎, ‎\\‎
‎g(t,s,y(s))&=&\bar{g}(t^{1/b},s^{1/b},y(s)),
‎\end{eqnarray*}‎
‎where $\bar{f}$ and $\bar{g}$ are analytic functions in the neighborhood of $(0,y_{0}^{(0)})$ and $(0,0,y_{0}^{(0)})$‎, ‎respectively‎. ‎Then the solution of \eqref{eq1} has the following series form in the neighborhood of the initial point‎
‎\begin{equation}\label{eq1500}‎
‎y(t)=\psi(t)+\sum\limits_{\mu=\alpha b}^{\infty}{\bar{y}_{\mu}~t^{\frac{\mu}{b}}},
‎\end{equation}‎
‎where $\bar{y}_{\mu}$ are known coefficients‎, and $b$ signifies the least common multiple of $b_{i}$‎, ‎$i=1,2$.
‎\end{theorem}‎
‎\begin{proof}‎
Regard the series expansion of $y(t)$ as
‎\begin{equation}\label{eq3}‎
‎y(t)=\sum\limits_{\mu=0}^{\infty}{\bar{y}_{\mu}~t^{\frac{\mu}{b}}}‎.
‎\end{equation}‎
‎The unknown coefficients $\bar{y}_{\mu}$ are obtained in order that the series \eqref{eq3} converges and solves \eqref{eq1}‎. ‎We use the series expansions of $\bar{f}$ and $\bar{g}$ around $(0,y_{0}^{(0)})$ and $(0,0,y_{0}^{(0)})$‎, ‎respectively‎, ‎viz‎.
‎\begin{eqnarray}\label{e54-11}‎
‎\nonumber f(t,y(t))&=&\bar{f}(t^{1/b},y(t))=\sum\limits_{\substack{\mu=0 \\ \theta=0}}^{\infty}{f_{\mu,\theta}~t^{\frac{\mu}{b}}\bigg(y(t)-y_{0}^{(0)}\bigg)^{\theta}}‎, ‎\\‎
‎g(t,s,y(s))&=&\bar{g}(t^{1/b},s^{1/b},y(s))=\sum\limits_{\substack{\mu,\nu=0 \\ \theta=0}}^{\infty}{g_{\mu,\nu,\theta}~t^{\frac{\mu}{b}}s^{\frac{\nu}{b}}\bigg(y(s)-y_{0}^{(0)}\bigg)^{\theta}}.
‎\end{eqnarray}‎

‎By rearranging‎, ‎it can be concluded that‎
‎\begin{equation}\label{e235}‎
‎\bigg(y(t)-y_{0}^{(0)}\bigg)^{\theta}=\bigg(\sum\limits_{\mu=1}^{\infty}{\bar{y}_{\mu} t^{\frac{\mu}{b}}}\bigg)^{\theta}=‎
‎\sum\limits_{\sigma=0}^{\infty}{C_{\sigma}^{\theta} t^{\frac{\sigma}{b}}}‎,
‎\end{equation}‎
‎where‎
‎\begin{equation*}
‎{C}_{\sigma}^{\theta}=\begin{cases}‎
‎1,\hspace{+4.5cm} \theta=0,~\sigma=0,\\‎
‎0,\hspace{+4.5cm} \theta=0,~\sigma \geq 1,\\‎
‎\sum\limits_{\mu_{1}+\ldots+\mu_{\theta}=\sigma}\bar{y}_{\mu_{1}}‎
‎\ldots \bar{y}_{\mu_{\theta}}, \quad \quad \quad \quad ~ \quad \theta \ne 0,~\sigma \geq 1‎.
‎\end{cases}‎
‎\end{equation*}‎

‎Therefore‎, in view of the equivalent equation \eqref{r3} and ‎substituting the relation \eqref{e235} into \eqref{e54-11}‎, ‎we find‎
‎\begin{multline*}
\sum\limits_{\mu=0}^{\infty}{\bar{y}_{\mu}~t^{\frac{\mu}{b}}}‎=\psi(t)+I^{\alpha}\Bigg(\sum\limits_{\substack{\mu=0 \\ \theta=0}}^{\infty}{f_{\mu,\theta}~t^{\frac{\mu}{b}}\Big(‎\sum\limits_{\sigma=0}^{\infty}{C_{\sigma}^{\theta} t^{\frac{\sigma}{b}}}\Big)‎}\Bigg)\\
+\lambda I^{\alpha}\Bigg(\int_{0}^{t}(t-s)^{{\beta}-1}\sum\limits_{\substack{\mu,\nu=0 \\ \theta=0}}^{\infty}{g_{\mu,\nu,\theta}~t^{\frac{\mu}{b}}s^{\frac{\nu}{b}}\Big(\sum\limits_{\sigma=0}^{\infty}{C_{\sigma}^{\theta} s^{\frac{\sigma}{b}}}\Big)ds}\Bigg).
\end{multline*}

‎Meanwhile‎, ‎by considering uniform convergence, ‎the coefficients $\bar{y}_{\mu}$ satisfy the following equality‎
‎‎\begin{equation}\label{eq148}
\sum\limits_{\mu=0}^{\infty}{\bar{y}_{\mu}~t^{\frac{\mu}{b}}}‎=\psi(t)+\varrho_{1}\Bigg(\sum\limits_{\substack{\mu=0 \\ \theta,\sigma=0}}^{\infty}{f_{\mu,\theta}~C_{\sigma}^{\theta} ~t^{\alpha+\frac{\mu+\sigma}{b}}}\Bigg)‎+\lambda \varrho_{2}\Bigg(\sum\limits_{\substack{\mu,\nu=0 \\ \theta,\sigma=0}}^{\infty}{g_{\mu,\nu,\theta}~C_{\sigma}^{\theta} ~t^{\alpha+\beta+\frac{\mu+\nu+\sigma}{b}}}\Bigg),
\end{equation}
‎in which
$$
\varrho_{1}=\frac{\Gamma(\frac{\mu+\sigma}{b}+1)}{\Gamma(\frac{\mu+\sigma}{b}+\alpha+1)},\quad \varrho_{2}=\frac{\Gamma(\beta)\Gamma(\frac{\nu+\sigma}{b}+1)\Gamma(\frac{\mu+\nu+\sigma}{b}+\beta+1)}{\Gamma(\frac{\nu+\sigma}{b}+\beta+1)\Gamma(\frac{\mu+\nu+\sigma}{b}+\alpha+\beta+1)}.
$$
Setting $\mu=\mu-\sigma -\alpha b$ and $\mu=\mu-\nu-\sigma-\alpha b-\beta b$ into the first and second series on the right-hand side of \eqref{eq148} respectively deduces
‎\begin{multline}\label{eq6}
\sum\limits_{\mu=0}^{\infty}{\bar{y}_{\mu}~t^{\frac{\mu}{b}}}‎=\psi(t)+\bar{\varrho}_{1}\Bigg(\sum\limits_{\substack{\mu=\sigma+\alpha b \\ \theta,\sigma=0}}^{\infty}{f_{\mu-\sigma-\alpha b,\theta}~C_{\sigma}^{\theta} ~t^{\frac{\mu}{b}}}\Bigg)‎\\
+\lambda \bar{\varrho}_{2}\Bigg(\sum\limits_{\substack{\mu=\nu+\sigma+\alpha b+\beta b \\ \nu,\theta,\sigma=0}}^{\infty}{g_{\mu-\nu-\sigma-\alpha b-\beta b,\nu,\theta}~C_{\sigma}^{\theta} ~t^{\frac{\mu}{b}}}\Bigg),
\end{multline}
where
$$
\bar{\varrho}_{1}=\frac{\Gamma(\frac{\mu}{b}-\alpha+1)}{\Gamma(\frac{\mu}{b}+1)},\quad \bar{\varrho}_{2}=\frac{\Gamma(\beta)\Gamma(\frac{\nu+\sigma}{b}+1)\Gamma(\frac{\mu}{b}-\alpha+1)}{\Gamma(\frac{\nu+\sigma}{b}+\beta+1)\Gamma(\frac{\mu}{b}+1)}.
$$
We now arrive at $\bar{y}_{\mu}$ through making a comparison between the coefficients of $t^{\frac{\mu}{b} }$ on both sides of \eqref{eq6}. ‎Evidently‎, ‎for $\mu < \alpha b$‎, ‎we can write‎
\begin{equation*}‎
\bar{y}_{\mu}=\left\{\begin{array}{l}‎
\frac{y_{0}^{(\frac{\mu}{b})}}{(\frac{\mu}{b})!},\quad \mu=0,b‎, ..., ‎(\lceil \alpha \rceil-1) b,\\‎
‎0,\hspace{2 cm}\text{else}‎,
\end{array}\right.
\end{equation*}‎

‎and for $\mu \geq \alpha b$‎, ‎the following recursive relation is derived‎
‎\[‎
\bar{y}_{\mu}‎=\sum\limits_{\substack{ \nu,\theta,\sigma=0}}^{\infty}{C_{\sigma}^{\theta}\Bigg(\bar{\varrho}_{1}~f_{\mu-\sigma-\alpha b,\theta}+\lambda \bar{\varrho}_{2}~g_{\mu-\nu-\sigma-\alpha b-\beta b,\nu,\theta}\Bigg)},
‎\]‎
‎such that the coefficients $f_{\mu-\sigma-\alpha b,\theta}$ and $g_{\mu-\nu-\sigma-\alpha b-\beta b,\nu,\theta}$, equipped with negative indices, are deemed to be zero‎, ‎and in the case of $\mu-\sigma-\alpha b \geq 0$ and $\mu-\nu-\sigma-\alpha b-\beta b \geq 0$‎, ‎we observe that‎
‎\begin{eqnarray*}‎
‎\mu &\geq &\sigma +\alpha b >\sigma=\mu_{1}+\ldots+\mu_{\theta}\geq \mu_{i}‎, \\
‎\mu &\geq& \nu+\sigma +\alpha b+\beta b >\sigma=\mu_{1}+\ldots+\mu_{\theta}\geq \mu_{i}‎, ‎\quad i=1,2,\ldots,\theta.
‎\end{eqnarray*}‎‎
‎In other words‎, ‎$\bar{y}_{\mu}$ are obtained recursively‎, and thereby the series expansion \eqref{eq1500} is a unique formal solution of \eqref{eq1}‎.

We are now required to verify that such power series is uniformly and absolutely convergent in the neighborhood of zero‎. In this regard, we apply Lindelof's majorant approach \cite{39}‎. ‎Deem the weakly singular Volterra integral equation
‎\[‎
Y(t)=\bar{\psi}(t)+I^{\alpha}F(t,y(t))+\lambda I^{\alpha}\Big(\int_{0}^{t}(t-s)^{{\beta}-1}G(t,s,Y(s))ds\Big),
‎\]‎
‎where‎ $\bar{\psi}(t)=\sum\limits_{k=0}^{\lceil \alpha \rceil -1}\frac{t^{k}}{k!} \lvert y_{0}^{(k)}\rvert$
‎\begin{eqnarray*}‎
‎\nonumber F(t,y(t))&=&\bar{F}(t^{1/b},Y(t))=\sum\limits_{\substack{ \mu=0 \\ \theta=0}}^{\infty}{\lvert f_{\mu,\theta}\rvert~t^{\frac{\mu}{b}}\bigg(Y(t)-\lvert y_{0}^{(0)}\rvert\bigg)^{\theta}}‎, ‎\\‎
‎G(t,s,Y(s))&=&\bar{G}(t^{1/b},s^{1/b},Y(s))=\sum\limits_{\substack{\mu, \nu=0 \\ \theta=0}}^{\infty}{\lvert g_{\mu,\nu,\theta} \rvert~t^{\frac{\mu}{b}}‎s^{\frac{\nu}{b}}‎
\bigg(Y(t)-\lvert y_{0}^{(0)}\rvert\bigg)^{\theta}}‎.
‎\end{eqnarray*}‎

‎It is clear that all coefficients of $Y$ are positive, and it is a majorant for $y$‎. ‎The series expansion $Y$ may be calculated in precisely the same manner as above‎. Currently‎, for some $\omega>0$ given in the sequel, we prove that the series $Y(t)$ converges absolutely over $[0,\omega]$. For this purpose‎, the key is to justify that the finite partial sum of the formal solution $Y(t)$ i.e.
‎\[‎
‎S_{L+1}(t)=\bar{\psi}(t)+\sum\limits_{\mu=\alpha b}^{L+1}{\bar{Y}_{\mu}~t^{\frac{\mu}{b}}}‎,
‎\]‎
can be uniformly bounded on $[0,\omega]$‎‎. ‎Clearly‎, ‎we have the inequality below
‎\begin{equation}‎\label{eqs254}
‎S_{L+1}(t) \le \bar{\psi}(t)+I^{\alpha}F(t,S_{L}(t))+\lambda I^{\alpha}\Big(\int_{0}^{t}(t-s)^{{\beta}-1}G(t,s,S_{L}(s))ds\Big),
‎\end{equation}‎
in accordance with the recursive evaluation of the coefficients‎. Albeit‎, ‎all the coefficients $\bar{Y}_{\mu}$ with $\frac{\mu}{b}\le \frac{(L+1)}{b}$ are removed from both sides by expanding the right-hand side of \eqref{eqs254}, there still exists some extra terms of higher-order in the right-hand side‎. Considering‎
‎\begin{eqnarray*}‎
‎D^{1}&=&\sum\limits_{k=0}^{\lceil \alpha \rceil -1}\frac{T^{k}}{k!}\lvert y_{0}^{(k)}\rvert,\\‎
‎D^{2}&=&\max_{(t,z)\in \Lambda \times [0‎
‎,3D^{1}]}{\frac{\big \lvert F(t,z)\big \rvert}{\Gamma(\alpha+1)}},\\
D^{3}&=&\max_{(t,s,z)\in \Lambda \times \Lambda \times [0‎
‎,3D^{1}]}{\frac{\Gamma(\beta)\big \lvert G(t,s,z)\big \rvert}{\Gamma(\alpha+\beta+1)}},
‎\end{eqnarray*}‎
‎we define‎
‎$$\omega=\min\bigg\{T,\bigg[\frac{D^{1}}{D^{2}}\bigg]^{\frac{1}{\alpha}},\bigg[\frac{D^{1}}{D^{3}}\bigg]^{\frac{1}{\alpha+\beta}}\bigg\}.$$‎
‎At this step, the aim is to show that $\lvert S_{L}(t) \rvert \le 3 D^{1}$, $t \in [0,\omega]$ by means of mathematical induction on $L$‎‎. The result is valid‎ for $L=0$‎ because
‎\[‎
‎S_{0}(t)=\lvert y_{0}^{(0)} \rvert \leq D^{1}.
‎\]‎

‎We regard that it is valid for $L$‎, ‎and proceed to $L+1$‎ as follows
‎\begin{eqnarray*}‎
‎\lvert S_{L+1}(t) \rvert&=&S_{L+1}(t) \\
&\le& \bar{\psi}(t)+I^{\alpha}F(t,S_{L}(t))+\lambda I^{\alpha}\Big(\int_{0}^{t}(t-s)^{{\beta}-1}G(t,s,S_{L}(s))ds\Big) \\‎
  ‎ &\le& \sum\limits_{k=0}^{\lceil \alpha \rceil -1}\frac{\omega^{k}}{k!}\lvert y_{0}^{(k)}\rvert + \max\limits_{s\in [0,t]}\lvert F(s,S_{L}(s))\rvert \frac{t^{\alpha}}{\Gamma(\alpha+1)}\\
  &+&\max\limits_{s\in [0,t]}\lvert G(t,s,S_{L}(s))\rvert \frac{\Gamma(\beta)~t^{\alpha+\beta}}{\Gamma(\alpha+\beta+1)}\\‎
  ‎ &\le& D^{1}+\max_{(s,z)\in [0,\omega] \times [0‎
,‎3D^{1}]}\lvert F(s,z)\rvert \frac{\omega^{\alpha}}{\Gamma(\alpha+1)}\\
&+&\max\limits_{(t,s,z)\in [0,\omega] \times [0,\omega] \times [0‎
,‎3D^{1}]}\lvert G(t,s,z)\rvert \frac{\Gamma(\beta)~\omega^{\alpha+\beta}}{\Gamma(\alpha+\beta+1)}\\‎
   &\le&  D^{1}+\omega^{\alpha} D^{2}+\omega^{\alpha+\beta} D^{3}\le 3 D^{1}‎.\hspace{+7.8cm}
\end{eqnarray*}‎
This establishes that $S_{L+1}(t)$ is uniformly bounded on $[0‎, \omega]$‎. Since all the coefficients $S_{L+1}(t)$ are positive, it is monotone as well‎. As a result‎, in according to the power series structure of $Y(t)$‎, the majorant $Y(t)$ converges absolutely over $[0,\omega]$, and it is uniformly convergent on the compact subsets of $[0,\omega)$‎. Ultimately‎, Lindelof's majorant theorem concludes that‎ the series expansion $y(t)$‎ inherits the same features. That is why the above exchange of integration and series was legal‎.
‎\end{proof}‎
Theorem \ref{th202} tells us that $\partial_t^{\lceil \alpha \rceil}  y(t)$ may not be continuous at the initial point. In consequence, the accuracy of the classical spectral methods might be threatened by this difficulty. That is why constructing an exponentially accurate or high-order spectral method is a kind of challenging task. Next section includes such approach to approximate the solution of \eqref{eq1} which satisfies the assumptions of theorems \ref{thm3} and \ref{th202}.
\section{Numerical approach}\label{sec2}
In this section, the goal is firstly to direct some critical properties of Jacobi polynomials, GJPs, and FGJFs, and present an effective strategy based on an advanced operational Petrov-Galerkin method to approximate the solution of \eqref{eq1}.
\subsection{Jacobi polynomials}
The orthogonal relation of classical Jacobi polynomials $J_n^{(\rho,\eta)}(s)$, $\rho, \eta >-1$ is as follows \cite{29}
\begin{equation*}
‎\int_{I}J_{m}^{(\rho,\eta)}(s)
J_{n}^{(\rho,\eta)}(s)w^{(\rho,\eta)}(s)ds=\lambda_{n}^{(\rho,\eta)}\delta_{mn}‎, ‎\quad  m,n\geq 0‎,
\end{equation*}
in which‎ $w^{(\rho,\eta)}(s)=(1-s)^\rho(1+s)^\eta$,
$$\lambda_{n}^{(\rho,\eta)}=\|J_{n}^{(\rho,\eta)}\|_{w^{(\rho,\eta)}}^{2}=\dfrac{2^{\rho+\eta+1}\Gamma(n+\rho+1)\Gamma(n+\eta+1)}
{(2n+\rho+\eta+1)n!\Gamma(n+\rho+\eta+1)},‎$$
‎and $\delta_{mn}$
 directs the Kronecker delta function. Jacobi polynomials satisfy the Rodrigues' formula below
\[
w^{(\rho,\eta)}(s)J_n^{(\rho,\eta)}(s)=\frac{(-1)^n}{2^n n!}\frac{d^n}{ds^n}\big\{(1-s)^{n+\rho}(1+s)^{n+\eta}\big\}.
\]
\subsection{Generalized Jacobi polynomials}
Let us first define $\hat{\rho},\hat{\eta}$ and $\tilde{\rho},\tilde{\eta}$ from $\rho,\eta$ as follows
\begin{equation*}
\hat{\rho}=
\begin{cases}
  -\rho, & \rho \leq -1, \\
  0, \quad & \rho > -1,
\end{cases}
\quad \quad \quad \quad \tilde{\rho}=
\begin{cases}
  -\rho, & \rho \leq -1, \\
  \rho, & \rho > -1,
\end{cases}
\end{equation*}
likewise for $\hat{\eta}$ and $\tilde{\eta}$.\\
For each $\rho,\eta \in \mathbb{Z}$, the GJPs eliminating the restriction $\rho, \eta >-1$ are defined by \cite{26rv}
\begin{equation*}
\bar{J}_{n}^{(\rho,\eta)}(s)=(1-s)^{\hat{\rho}}(1+s)^{\hat{\eta}} J_{\tilde{n}}^{(\tilde{\rho},\tilde{\eta})}(s),\quad \tilde{n}=n-\kappa_{{\rho},{\eta}} \geq 0, ~~ \kappa_{{\rho},{\eta}}=\hat{\rho}+\hat{\eta}.
\end{equation*}

The GJPs satisfy the Sturm-Liouville equation
    \begin{equation*}
        \partial_s \left((1-s)^{\rho+1}(1+s)^{\eta+1} \partial_s \bar{J}_n^{(\rho,\eta)}(s)\right)+\xi_{n}^{(\rho,\eta)}(1-s)^{\rho}(1+s)^{\eta}\bar{J}_n^{(\rho,\eta)}(s)=0,
    \end{equation*}
    where
    \begin{equation*}
      \xi_{n}^{(\rho,\eta)}=\begin{cases}
      \tilde{n}(\tilde{n}+\rho+\eta+1), & (\rho,\eta) > -1, \\
      \tilde{n}(\tilde{n}-\rho+\eta+1)-\rho(\eta+1), &  \rho \leq -1, ~\eta > -1,\\
      \tilde{n}(\tilde{n}+\rho-\eta+1)-\eta(\rho+1), & \rho > -1, ~\eta \leq -1,\\
      (\tilde{n}+1)(\tilde{n}-\rho-\eta), & (\rho,\eta) \leq -1,
      \end{cases}
    \end{equation*}
    and are mutually $L_{w^{(\rho,\eta)}}^2(I)$-orthogonal, i.e.,
\begin{equation*}
‎\int_{I}\bar{J}_{m}^{(\rho,\eta)}(s)\bar{J}_{n}^{(\rho,\eta)}(s)w^{(\rho,\eta)}(s)ds=\lambda_{\tilde{n}}^{(\tilde{\rho},\tilde{\eta})}\delta_{mn}‎, ‎\quad  m,n\geq \kappa_{{\rho},{\eta}}.
‎\end{equation*}‎

The following relations indicate outstanding properties of the GJPs
\begin{eqnarray*}
&&\partial_{s}^{i} \bar{J}_{n}^{(\rho,\eta)}(1)=0, \quad \quad \quad \quad  i=0,1,\ldots,‎ \hat{\rho} -1, \quad \quad \rho \leq -1, ~\eta >-1, \\
&&\partial_{s}^{j} \bar{J}_{n}^{(\rho,\eta)}(-1)=0,~~\quad \quad  j=0,1,\ldots,‎ \hat{\eta} -1, \quad \quad ~ \rho > -1, ~\eta \leq -1,
\end{eqnarray*}
and for $(\rho,\eta) \leq -1$, we have
\begin{eqnarray*}
&&\partial_{s}^{i} \bar{J}_{n}^{(\rho,\eta)}(1)=0, \quad \quad \quad \quad  i=0,1,\ldots,‎ \hat{\rho} -1,\\
&&\partial_{s}^{j} \bar{J}_{n}^{(\rho,\eta)}(-1)=0,~~~\quad \quad  j=0,1,\ldots, \hat{\eta} -1.
\end{eqnarray*}

This feature invokes the GJPs toward appropriate basis functions for the Galerkin and Petrov-Galerkin approximations of smooth solutions of functional differential equations associated with the following boundary conditions
\begin{eqnarray*}
\partial_{s}^{i} P(1)&=&0, \quad \quad \quad \quad  i=0,1,\ldots, \hat{\rho} -1,\\
\partial_{s}^{j} P(-1)&=&0, \quad \quad \quad \quad  j=0,1,\ldots, \hat{\eta} ‎-1.
\end{eqnarray*}
\subsection{The fractional generalized Jacobi functions}
Let us begin this section with a vital question. Is it feasible to promote the GJPs, whereby a novel set of fractional orthogonal functions and related Galerkin and Petrov-Galerkin approximations are achieved? Thankfully, Faghih and Mokhtary \cite{sevom} introduced the fractional generalized Jacobi functions which have outstanding approximate properties to the functions owning singularity at boundaries.

Suppose that $\gamma \in (0,1]$ and $t \in \Lambda$, the essence of the FGJFs $\bar{J}_n^{(\rho,\eta,\gamma)}(t)$ comes from the GJPs by means of the coordinate transformation $s=2 (\frac{t}{T})^{\gamma}-1$ as \cite{sevom}
\begin{equation}\label{e6}
\bar{J}_n^{(\rho,\eta,\gamma)}(t)=\bar{J}_n^{(\rho,\eta)}\left(2 (\frac{t}{T})^{\gamma}-1\right)=\frac{2^{\hat{\rho}+\hat{\eta}}}{T^{\gamma(\hat{\rho}+\hat{\eta})}}
\bigg(T^\gamma-t^\gamma\bigg)^{\hat{\rho}} t^{\gamma \hat{\eta}} J_{\tilde{n}}^{(\tilde{\rho},\tilde{\eta},\gamma)}(t),
\end{equation}
in which $\rho, \eta \in \mathbb{Z}$, $n \ge \kappa_{\rho,\eta}$ and
\begin{equation*}
J_{\tilde{n}}^{(\tilde{\rho},\tilde{\eta},\gamma)}(t)=J_{\tilde{n}}^{(\tilde{\rho},\tilde{\eta})}\left(2 (\frac{t}{T})^{\gamma}-1\right),
\end{equation*}
indicates the fractional Jacobi function (FJF) satisfying the explicit formula \cite{35}
‎\begin{equation*}‎
‎J_{\tilde{n}}^{(\tilde{\rho},\tilde{\eta},\gamma)}(t)=\sum_{j=0}^{\tilde{n}}\frac{1}{T^{\gamma j}}\Upsilon_{j}^{(\tilde{\rho},\tilde{\eta},\tilde{n})}t^{\gamma j}=Span\{1,t^{\gamma},\ldots,t^{\tilde{n}\gamma}\}‎,
\end{equation*}‎
where
‎\begin{equation*}‎
\Upsilon_{j}^{(\tilde{\rho},\tilde{\eta},\tilde{n})}=\dfrac{(-1)^{\tilde{n}-j}\Gamma(\tilde{n}+\tilde{\eta}+1)\Gamma(\tilde{n}+\tilde{\rho}+\tilde{\eta}+j+1)}{\Gamma(\tilde{\eta}+j+1)j!\Gamma(\tilde{n}+\tilde{\rho}+\tilde{\eta}+1)(\tilde{n}-j)!}‎.
‎\end{equation*}‎
It can be inferred that the FJFs
are $L_{w^{(\tilde{\rho},\tilde{\eta},\gamma)}}^2(\Lambda)$-orthogonal along with the weight function $w^{(\tilde{\rho},\tilde{\eta},\gamma)}(t)=\frac{\gamma}{T^{\gamma(\tilde{\rho}+\tilde{\eta}+1)-1}}\bigg(T^\gamma-t^\gamma\bigg)
^{\tilde{\rho}}t^{\gamma(\tilde{\eta}+1)-1}$, i.e., we have
\begin{equation*}
‎\int_{\Lambda}J_{\tilde{m}}^{(\tilde{\rho},\tilde{\eta},\gamma)}(t)J_{\tilde{n}}^{(\tilde{\rho},\tilde{\eta},\gamma)}(t)
w^{(\tilde{\rho},\tilde{\eta},\gamma)}(t)dt=\lambda_{\tilde{n}}^{(\tilde{\rho},\tilde{\eta},\gamma)}\delta_{\tilde{m}\tilde{n}}‎‎,
‎\end{equation*}‎
‎in which‎
‎\begin{equation*}‎
\lambda_{\tilde{n}}^{(\tilde{\rho},\tilde{\eta},\gamma)}=\|J_{\tilde{n}}^{(\tilde{\rho},\tilde{\eta},\gamma)}\|_
{w^{(\tilde{\rho},\tilde{\eta},\gamma)}}^{2}=
\frac{T}{2^{\tilde{\rho}+\tilde{\eta}+1}}\lambda_{\tilde{n}}^{(\tilde{\rho},\tilde{\eta})}.
\end{equation*}‎

The FGJFs $\{\bar{J}_n^{(\rho,\eta,\gamma)}\}_{n \geq \kappa_{\rho,\eta}}$ satisfies the Sturm-Liouville problem
\[
\partial_t \left(T(T^\gamma-t^\gamma) t^{2-\gamma} w^{(\rho,\eta,\gamma)}(t)‎ \partial_t \bar{J}_n^{(\rho,\eta,\gamma)}(t)\right)
+\frac{\gamma^2}{T} \xi_{n}^{(\rho,\eta)}w^{(\rho,\eta),\gamma}(t)‎\bar{J}_n^{(\rho,\eta,\gamma)}(t)=0.
\]

The FGJFs are mutually orthogonal along with the following orthogonal relation
\begin{equation*}
‎\int_{\Lambda}\bar{J}_{m}^{(\alpha,\beta,\gamma)}(t)\bar{J}_{n}^{(\alpha,\beta,\gamma)}(t)w^{(\alpha,\beta,\gamma)}(t)dt=\frac{T}{2^{\alpha+\beta+1}}
\lambda_{\tilde{n}}^{(\tilde{\alpha},\tilde{\beta})}\delta_{mn}‎.
\end{equation*}‎

The following relations hold
\begin{eqnarray*}
     \nonumber
      \partial_{t}^{i} \bar{J}_{n}^{(\rho,\eta,\gamma)}(T)&=&0, \quad \quad \quad \quad  i=0,1,\ldots,‎ \hat{\rho} ‎-1, \quad \quad \quad \rho \leq -1, ~\eta >-1, \\
     \nonumber \partial_{t}^{j} \bar{J}_{n}^{(\rho,\eta,\gamma)}(0)&=&0, \quad \quad \quad \quad  j=0,1,\ldots,\lceil‎ \gamma\hat{\eta} ‎\rceil-1, ~~ \quad \rho > -1, ~\eta \leq -1,\\
    \end{eqnarray*}
and for $\rho,\eta \leq -1$, we have
         \begin{eqnarray*}
     \nonumber
      \partial_{t}^{i} \bar{J}_{n}^{(\rho,\eta,\gamma)}(T)&=&0, \quad \quad \quad \quad  i=0,1,\ldots, \hat{\rho} -1,\\
     \partial_{t}^{j} \bar{J}_{n}^{(\rho,\eta,\gamma)}(0)&=&0, \quad \quad \quad \quad  j=0,1,\ldots,\lceil‎ \gamma \hat{\eta} ‎\rceil-1,
    \end{eqnarray*}
which is one of the remarkable features of the FGJFs.
\subsection{Operational Petrov-Galerkin method}
This section offers a highly accurate operational Petrov-Galerkin method based on the FGJFs to approximate the solution of \eqref{eq1}. Solvability analysis of the relevant non-linear algebraic system is also provided through a sequence of matrix operations. Inserting $\rho=0,~\eta=-\alpha b$ and $\gamma=\frac{1}{b}$ in \eqref{e6}, the FGJFs $\{\bar{J}_k^{0,-\alpha b,\gamma}\}_{k \geq \alpha b}$
\begin{equation}\label{ert25}
\bar{J}_k^{(0,-\alpha b,\gamma)}(t)=\frac{2^{\alpha b}}{T^\alpha}t^\alpha J_{k-\alpha b}^{(0,\alpha b,\gamma)}(t)=\bar{J}_k^{(0,-\alpha b,\gamma)}=\text{Span}\{t^{\alpha},t^{\alpha+\gamma},\ldots,t^{k \gamma}\},
\end{equation}
satisfy the initial conditions of the equation \eqref{eq1}. Thereby they are capable of serving as basis functions to obtain the Petrov-Galerkin approximation of \eqref{eq1}.

Based on \eqref{ert25}, we arrive at $\underline{J}^{(0,-\alpha b,\gamma)}=J \underline{T}_{t}$, where
$$\underline{J}^{(0,-\alpha b,\gamma)}=[\bar{J}_{\alpha b}^{(0,-\alpha b,\gamma)}(t),\bar{J}_{\alpha b+1}^{(0,-\alpha b,\gamma)}(t)‎,
‎\ldots,\bar{J}_{N}^{(0,-\alpha b,\gamma)}(t),\ldots]^{\prime},$$
indicating the vector of FGJFs with degree $(\bar{J}_{k}^{(0,-\alpha b,\gamma)}(t))\leq k \gamma$, $\underline{T}_{t}=[t^{\alpha},t^{\alpha+\gamma},\ldots,t^{N\gamma},\ldots]^{\prime}$ ‎and $J$ is a lower-triangular matrix of order infinity. At this stage, the Petrov-Galerkin approximation $y_{N}(t)$ of the exact solution $y(t)$ is stated as
\begin{equation}\label{e31}
y_{N}(t)=\sum\limits_{k=0}^{\infty}{c_{k}\bar{J}_{k+\alpha b}^{(0,-\alpha b,\gamma)}(t)}=\underline{c}~\underline{J}^{(0,-\alpha b,\gamma)}=
\underline{c} J \underline{T}_t=\underline{c} J \Psi \underline{\hat{T}}_{t},
\end{equation}
where $\underline{c}=[c_{0},c_{1},\ldots,c_{\hat{N}},0,\ldots]‎, ~\underline{\hat{T}}_{t}=[1, t^{\gamma}, \ldots, t^{\hat{N}\gamma}, \ldots]^{\prime},~\hat{N}=N-\alpha b$, and
\begin{eqnarray}‎\label{e46-2}
\Psi=‎
‎\begin{bmatrix}‎
‎\overbrace{0 \ldots  0}^{\alpha b}&1&0&\cdots& &\\‎
‎\vdots&0&1&0&\cdots&\\‎
‎\vdots&\vdots&0&1&0&\cdots\\
‎\cdots&\cdots&&\ddots&\ddots&\ddots\\
‎\end{bmatrix}‎.
\end{eqnarray}
Meanwhile, the relation \eqref{e54-11} can be restated as
\begin{eqnarray}\label{e33}
‎\nonumber f(t,y(t))&=&\sum\limits_{\substack{\mu=0 \\ \theta=0}}^{\infty}{f_{\mu,\theta}~t^{\mu \gamma}~y^{\theta}(t)}‎, ‎\\‎
‎g(t,s,y(s))&=&\sum\limits_{\substack{\mu,\nu=0 \\ \theta=0}}^{\infty}{g_{\mu,\nu,\theta}~t^{\mu \gamma}s^{\nu \gamma}~y^{\theta}(s)}.
\end{eqnarray}

Inserting the relation \eqref{e31} into \eqref{eq1} and using \eqref{e33}, we have
\begin{equation}\label{e34}
‎\underline{c} J D^{\alpha}_C \underline{T}_t=\sum\limits_{\substack{\mu=0 \\ \theta=0}}^{\infty}{f_{\mu,\theta}~t^{\mu \gamma}~y_{N}^{\theta}(t)}‎+\lambda \sum\limits_{\substack{\mu,\nu=0 \\ \theta=0}}^{\infty}g_{\mu,\nu,\theta}~t^{\mu \gamma} \int_{0}^{t}(t-s)^{{\beta}-1}s^{\nu \gamma}~y_{N}^{\theta}(s)ds.
\end{equation}
Hence, we first compute $ D^{\alpha}_C \underline{T}_{t}$. In this respect, using the relation \cite{39,30,31} $$D^{\alpha}_Ct^{\tau}=\frac{\Gamma(\tau+1)}{\Gamma(\tau-\alpha+1)}t^{\tau-\alpha},~~\tau>0,$$ we can write
\begin{equation}\label{e23}
D^{\alpha}_C \underline{T}_{t}=\Bigg[D^{\alpha}_C t^{\alpha+\gamma i}\Bigg]_{i \ge 0}=
\Bigg[\frac{\Gamma(\alpha+\gamma i+1)}{\Gamma(\gamma i+1)}t^{\gamma i}\Bigg]_{i \ge 0}=\chi \underline{\hat{T}}_{t},
\end{equation}
in which
\[
\chi =
\begin{bmatrix}
\Gamma{(\alpha+1)}&0&0&\ldots\\
0&\frac{\Gamma{(\alpha+\gamma+1)}}{\Gamma{(\gamma+1)}}&0&\cdots\\
0&0&\frac{\Gamma{(\alpha+2\gamma+1)}}{\Gamma{(2\gamma+1)}}&\cdots\\
\vdots&\vdots&0&\ddots \\
\end{bmatrix}.
\]
The key of this strategy is to attain a matrix form for the right-hand side of \eqref{e34} vigorously. In this sense, we first substitute an appropriate matrix representation for $y_{N}^{\theta}(t)$ through the following lemma.
\begin{lemma}\label{lem31}
Assuming $\underline{\underline{c}}=\underline{c} J=[\underline{c}_{0},\underline{c}_{1},\ldots,\underline{c}_{\hat{N}},0,\ldots]$, the following relation holds
\begin{equation*}
y_{N}^{\theta}(t)=\underline{\underline{c}}~\Psi Q^{\theta-1}\underline{\hat{T}}_{t},\quad \theta > 0,
\end{equation*}
where $Q$ indicates the following upper-triangular matrix of order infinity
\[
Q=
\begin{bmatrix}
\underline{\underline{c}} \Psi_{0}&\underline{\underline{c}} \Psi_{1}&\underline{\underline{c}} \Psi_{2}&\ldots\\
0&\underline{\underline{c}} \Psi_{0}&\underline{\underline{c}} \Psi_{1}&\ldots\\
0&0&\underline{\underline{c}} \Psi_{0}&\ldots \\
\vdots&\vdots&\vdots&\ddots
\end{bmatrix},
\]
with $\Psi_{r}=\{\Psi_{m,r}\}_{m=0}^{\infty},~r=0,1,\ldots~$.
\end{lemma}
\begin{proof}
The mathematical induction on $\theta$ is utilized to prove this lemma. For $\theta=1$, the result is obviously valid. We consider that it is valid for $\theta$, and proceed to $\theta+1$ as follows
\begin{eqnarray}\label{e36}
\nonumber y_{N}^{\theta+1}(t)&=&y_{N}^{\theta}(t)\times y_{N}(t)=\big(\underline{\underline{c}}~ \Psi Q^{\theta-1}\underline{\hat{T}}_{t}\big)\times(\underline{\underline{c}} ~\Psi \underline{\hat{T}}_{t})\\
&=&\underline{\underline{c}}~ \Psi Q^{\theta-1}( \underline{\hat{T}}_{t}\times(\underline{\underline{c}} ~\Psi \underline{\hat{T}}_{t})).
\end{eqnarray}

Next, it suffices to demonstrate that
\begin{equation}\label{e35}
 \underline{\hat{T}}_{t}\times(\underline{\underline{c}} ~\Psi \underline{\hat{T}}_{t})=Q \underline{\hat{T}}_{t}.
\end{equation}

For this purpose, we can derive
\begin{eqnarray*}
\nonumber \underline{\hat{T}}_{t}\times(\underline{\underline{c}} ~\Psi \underline{\hat{T}}_{t})&=&\underline{\hat{T}}_{t}) \times \Big(\sum\limits_{h=0}^{\infty}{\sum\limits_{k=0}^{\infty}{\underline{c}_{k}~\Psi_{k,h}~t^{h\gamma}}}\Big)=\Bigg[
\sum\limits_{h=0}^{\infty}{\sum\limits_{k=0}^{\infty}{\underline{c}_{k}~\Psi_{k,h}~t^{(h+m)\gamma}}}\Bigg]_{m=0}^{\infty}\\
&=&\Bigg[\sum\limits_{r=m}^{\infty}{\Big(\sum\limits_{k=0}^{\infty}{\underline{c}_{k}~\Psi_{k,r-m}~t^{r\gamma}}\Big)}\Bigg]_{m=0}^{\infty},
\end{eqnarray*}
whereby it yields \eqref{e35}. Trivially, by inserting \eqref{e35} into \eqref{e36} the desired result can be achieved.
\end{proof}
Now, it is time to employ Lemma \ref{lem31} to convert the first term of the right-hand side of \eqref{e34} into the following vector-matrix form
\begin{eqnarray}\label{e40}
\nonumber \sum\limits_{\substack{\mu=0 \\ \theta=0}}^{\infty}{f_{\mu,\theta}~t^{\mu \gamma}~y_{N}^{\theta}(t)}‎&=&\sum\limits_{\mu=0 }^{\infty}{f_{\mu,0}~t^{\mu \gamma}}+\sum\limits_{\substack{\mu=0 \\ \theta=1}}^{\infty}{f_{\mu,\theta}~t^{\mu \gamma}~y_{N}^{\theta}(t)}\\
\nonumber&=&\underline{f}_{0}\underline{\hat{T}}_{t}+\sum\limits_{\substack{\mu=0 \\ \theta=1}}^{\infty}{f_{\mu,\theta}~t^{\mu \gamma}\underline{\underline{c}}~\Psi Q^{\theta-1}\underline{\hat{T}}_{t}}\\
\nonumber &=&\underline{f}_{0}\underline{\hat{T}}_{t}+\sum\limits_{\theta=1 }^{\infty}{\underline{\underline{c}}~\Psi Q^{\theta-1}\Bigg(\sum\limits_{\mu=0 }^{\infty}f_{\mu,\theta}~\left[t^{(\mu+m) \gamma}\right]_{m=0}^{\infty}\Bigg)}\\
&=&\Bigg(\underline{f}_{0}+\sum\limits_{\theta=1 }^{\infty}{\underline{\underline{c}}~\Psi Q^{\theta-1}F_{\theta}}\Bigg)\underline{\hat{T}}_{t},
\end{eqnarray}
in which $\underline{f}_{0}=[f_{0,0},f_{1,0},\ldots,f_{\hat{N},0},\ldots]$ and $F_{\theta}$ directs an upper-triangular matrix of order infinity with the following components
\begin{equation*}
\big[F_{\theta}\big]_{i,j=0}^{\infty}=\begin{cases}
0, \quad \quad \quad \quad ~~ i > j,\\
f_{j-i,\theta},\quad \quad ~~ i \leq j.
\end{cases}
\end{equation*}
Applying the same strategy to the second term of \eqref{e34} concludes
\begin{multline}\label{e40-1}
\sum\limits_{\substack{\mu,\nu=0 \\ \theta=0}}^{\infty}g_{\mu,\nu,\theta}~t^{\mu \gamma} \int_{0}^{t}(t-s)^{{\beta}-1}s^{\nu \gamma}~y_{N}^{\theta}(s)ds\\
=\sum\limits_{\mu,\nu=0 }^{\infty}g_{\mu,\nu,0}~t^{\mu \gamma} \int_{0}^{t}(t-s)^{{\beta}-1}s^{\nu \gamma}ds+\sum\limits_{\substack{\mu,\nu=0 \\ \theta=1}}^{\infty}g_{\mu,\nu,\theta}~t^{\mu \gamma} \int_{0}^{t}(t-s)^{{\beta}-1}s^{\nu \gamma}~\underline{\underline{c}}~\Psi Q^{\theta-1}\underline{\hat{T}}_{s}ds\\
\hspace{-6cm}=\sum\limits_{\mu,\nu=0 }^{\infty}g_{\mu,\nu,0}~t^{\mu \gamma} \int_{0}^{t}(t-s)^{{\beta}-1}s^{\nu \gamma}ds\\
+\sum\limits_{ \theta=1}^{\infty}\underline{\underline{c}}~\Psi Q^{\theta-1}\Bigg(\sum\limits_{ \mu,\nu=0}^{\infty}g_{\mu,\nu,\theta}~t^{\mu \gamma} \left[\int_{0}^{t}(t-s)^{{\beta}-1}s^{(\nu+m) \gamma}ds\right]_{m=0}^{\infty}\Bigg).
\end{multline}
Evidently, we have \cite{39}
\begin{eqnarray*}
\left[\int_{0}^{t}(t-s)^{{\beta}-1}s^{(\nu+m) \gamma}ds\right]_{m=0}^{\infty}&=&\left[\mathcal{A}_{\nu}^{m}~t^{(\nu+m) \gamma+\beta}\right]_{m=0}^{\infty},\\
\int_{0}^{t}(t-s)^{{\beta}-1}s^{\nu \gamma}ds&=&\mathcal{A}_{\nu}^{0}~t^{\nu \gamma+\beta},
\end{eqnarray*}
in which $\mathcal{A}_{\nu}^{m}=\frac{\Gamma(\beta)\Gamma((\nu+m) \gamma+1)}{\Gamma((\nu+m) \gamma+\beta+1)}$. Therefore, the equality \eqref{e40-1} can be rewritten as
\begin{multline}\label{e41}
\sum\limits_{\substack{\mu,\nu=0 \\ \theta=0}}^{\infty}g_{\mu,\nu,\theta}~t^{\mu \gamma} \int_{0}^{t}(t-s)^{{\beta}-1}s^{\nu \gamma}~y_{N}^{\theta}(s)ds\\
=\sum\limits_{\mu,\nu=0 }^{\infty}g_{\mu,\nu,0}~\mathcal{A}_{\nu}^{0}~t^{(\mu +\nu)\gamma+\beta}+\sum\limits_{ \theta=1}^{\infty}\underline{\underline{c}}~\Psi Q^{\theta-1}\Bigg(\sum\limits_{ \mu,\nu=0}^{\infty}g_{\mu,\nu,\theta} \left[\mathcal{A}_{\nu}^{m}~t^{(\mu+\nu+m) \gamma+\beta}\right]_{m=0}^{\infty}\Bigg)\\
=\Bigg(\underline{K}+\sum\limits_{ \theta=1}^{\infty}\underline{\underline{c}}~\Psi Q^{\theta-1}H_{\theta}\Bigg)\underline{\hat{T}}_{t},\hspace{+7cm}
\end{multline}
along with the infinite-order row vector $\underline{K}$ as
\begin{eqnarray*}
\nonumber \big[\underline{K}\big]_{j=0}^{\infty}&=&\begin{cases}
0,\hspace{+4.2cm}\quad  j < \beta b,\\
\sum\limits_{\nu=0}^{j-\beta b}g_{\nu,j-\beta b -\nu,0}~\mathcal{A}_{j-\beta b-\nu}^{0},\quad \quad \quad ~ j \geq \beta b,
\end{cases}
\\&=&[\overbrace{0 \ldots  0}^{\beta  b},\big[\underline{K}\big]_{\beta b},
\big[\underline{K}\big]_{\beta b+1},\ldots,\big[\underline{K}\big]_{\hat{N}},\ldots],
\end{eqnarray*}
and the infinite-order upper-triangular matrix $H_{\theta}$ with the components
\begin{equation*}
\big[H_{\theta}\big]_{i,j=0}^{\infty}=\begin{cases}
0,\hspace{4.9cm} \quad ~ i \geq j-\beta b+1,\\
\sum\limits_{\nu=0}^{j-i-\beta b}g_{\nu,j-i-\beta b-\nu,\theta}~\mathcal{A}_{j-i-\beta b-\nu}^{i},\quad \quad  \quad ~~ i <j- \beta b+1.
\end{cases}
\end{equation*}

At this stage, we employ the relations \eqref{e23}, \eqref{e40} and \eqref{e41} in \eqref{e34}, and thereby we derive
\begin{equation}\label{e42}
\underline{\underline{c}}~\chi \underline{\hat{T}}_{t}=\Bigg(\underline{f}_{0}+\sum\limits_{\theta=1 }^{\infty}{\underline{\underline{c}}~\Psi Q^{\theta-1}F_{\theta}}\Bigg)\underline{\hat{T}}_{t}+\lambda \Bigg(\underline{K}+\sum\limits_{ \theta=1}^{\infty}\underline{\underline{c}}~\Psi Q^{\theta-1}H_{\theta}\Bigg)\underline{\hat{T}}_{t}.
\end{equation}
Due to the orthogonality of $\{J_k^{(0,\alpha b,\gamma)}\}_{k \geq 0}$, we project \eqref{e42} onto $\{J_k^{(0,\alpha b,\gamma)}\}_{k =0}^{\hat{N}}$. Ultimately, after some simple manipulations, the algebraic form of the Petrov-Galerkin discretization is obtained as
\begin{equation}\label{e44}
\underline{\underline{c}}^{\hat{N}}~\chi^{\hat{N}} =\underline{f}_{0}^{\hat{N}}+\lambda \underline{K}^{\hat{N}}+\sum\limits_{ \theta=1}^{\infty}\underline{\underline{c}}^{\hat{N}}~\Psi^{\hat{N}} (Q^{\theta-1})^{\hat{N}}\Big(F_{\theta}^{\hat{N}}+\lambda H_{\theta}^{\hat{N}}\Big).
\end{equation}
Here, the sign $\hat{N}$ above the matrices and vectors signifies respectively the principle sub-matrices and sub-vectors of order $\hat{N}+1$. Needless to mention, we are able to get the unknown vector through solving the system of $\hat{N} + 1$ non-linear algebraic equations \eqref{e44}. The next section presents an outstanding strategy to overcome this system regardless of its complexity.
\subsection{Solvability analysis}
It is noteworthy that the system of non-linear algebraic equations \eqref{e44} involves high computational costs to be solved, specifically for large degrees of approximation, and it can undoubtedly result in inaccurate solution due to its complexity. In order to cope with this barrier, we aim to provide a productive and well-conditioned implementation that gives the unknown of \eqref{e44} by means of some recurrence relations instead of solving a complex non-linear algebraic system. For this purpose, applying Lemma \ref{lem31} enables us to write
\[
Q=
\begin{bmatrix}
\underline{\underline{c}} \Psi_{0}&\underline{\underline{c}} \Psi_{1}&\underline{\underline{c}} \Psi_{2}&\ldots\\
0&\underline{\underline{c}} \Psi_{0}&\underline{\underline{c}} \Psi_{1}&\ldots\\
0&0&\underline{\underline{c}} \Psi_{0}&\ldots \\
\vdots&\vdots&\vdots&\ddots
\end{bmatrix}=\begin{bmatrix}
‎\overbrace{0 \ldots  0}^{\alpha b}&\underline{c}_{0}&\underline{c}_{1}&\underline{c}_{2}&\ldots\\
\vdots&0&\underline{c}_{0}&\underline{c}_{1}&\ldots\\
\vdots&\vdots&0&\underline{c}_{0}&\ldots \\
\cdots&\cdots&\ddots&\ddots&\ddots
\end{bmatrix}.
\]
Through simple calculations, we observe that $(Q^{\theta-1})^{\hat{N}}$ has the following upper-triangular structure
\begin{eqnarray}\label{e46}
(Q^{\theta-1})^{\hat{N}}=
\nonumber && \begin{bmatrix}
‎\overbrace{0 \ldots  0}^{(\theta-1)\alpha b}&\left(\underline{c}_{0}\right)^{\theta-1}&(\theta-1)\left(\underline{c}_{0}\right)^{\theta-2}\underline{c}_{1}
&\ldots&\ldots\\
\vdots&0&\left(\underline{c}_{0}\right)^{\theta-1}&(\theta-1)\left(\underline{c}_{0}\right)^{\theta-2}\underline{c}_{1}&\ldots\\
\vdots&\vdots&\vdots&\ddots&\ddots \\
0&0&\ddots&0&\left(\underline{c}_{0}\right)^{\theta-1}\\
\vdots&\vdots&\vdots&\ddots&\vdots\\
0&0&0&0&0
\end{bmatrix}\\
\nonumber && =\begin{bmatrix}
\overbrace{0 \ldots  0}^{(\theta-1)\alpha b}&Q_{0,0}^{\theta-1}&Q_{0,1}^{\theta-1}&Q_{0,2}^{\theta-1}&\ldots\\
\vdots&0&Q_{0,0}^{\theta-1}&Q_{0,1}^{\theta-1}&\ldots&\\
\vdots&\vdots&\vdots&\ddots&\ddots\\
0&0&\ddots&0&Q_{0,0}^{\theta-1}\\
\vdots&\vdots&\vdots&\ddots&\vdots\\
0&0&0&0&0
\end{bmatrix},
\end{eqnarray}
where $Q_{0,r}^{\theta-1}$, $r=0,1,\ldots,N-\theta \alpha b$ are non-linear functions of the elements $\underline{c}_{0},~\underline{c}_{1},\ldots,~\underline{c}_{r}$.

In addition, from \eqref{e46-2} and \eqref{e46}, the following upper-triangular matrix of order $\hat{N}+1$ can be derived
\begin{eqnarray*}
\Psi^{\hat{N}} (Q^{\theta-1})^{\hat{N}}
\nonumber && =\begin{bmatrix}
\overbrace{0 \ldots  0}^{\theta \alpha b}&Q_{0,0}^{\theta-1}&Q_{0,1}^{\theta-1}&Q_{0,2}^{\theta-1}&\ldots\\
\vdots&0&Q_{0,0}^{\theta-1}&Q_{0,1}^{\theta-1}&\ldots&\\
\vdots&\vdots&\vdots&\ddots&\ddots\\
0&0&\ddots&0&Q_{0,0}^{\theta-1}\\
\vdots&\vdots&\vdots&\ddots&\vdots\\
0&0&0&0&0
\end{bmatrix}.
\end{eqnarray*}

Consequently, one can be checked that the matrix $\Pi$ defined by
\begin{equation*}
\Pi=\Psi^{\hat{N}} (Q^{\theta-1})^{\hat{N}} \mathcal{B},\quad \mathcal{B}=F_{\theta}^{\hat{N}}+\lambda H_{\theta}^{\hat{N}},
\end{equation*}
has an upper-triangular structure with the components below
\begin{eqnarray*}
\big[\Pi\big]_{i,j=0}^{\hat{N}}=\begin{cases}
0,\hspace{4.65cm} \quad  i \geq j-\theta \alpha b+1,\\
\sum\limits_{r=0}^{j-i-\theta \alpha b}Q_{0,r}^{\theta-1}~\big[\mathcal{B}\big]_{i+r+\theta \alpha b,j},\quad \quad \quad \quad ~i < j-\theta \alpha b+1.
\end{cases}
\end{eqnarray*}

Afterward, considering the structure of $\Pi$, we have
\begin{eqnarray*}
\underline{\underline{c}}^{\hat{N}} \Pi
=\underline{\underline{c}}^{\hat{N}}\begin{bmatrix}
\overbrace{0 \ldots  0}^{\theta \alpha b}&\big[  \Pi\big]_{0,\theta \alpha b}&\big[ \Pi\big]_{0,\theta \alpha b+1}&\ldots&\big[ \Pi\big]_{0,\hat{N}}\\
\vdots&0&\big[  \Pi\big]_{1,\theta \alpha b+1}&\cdots&\big[  \Pi\big]_{1,\hat{N}}\\
\vdots&\vdots&\vdots&\ddots&\ddots \\
0&0&\ddots&0&\big[  \Pi\big]_{\hat{N}-\theta \alpha b,\hat{N}}\\
\vdots&\vdots&\vdots&\ddots&\vdots\\
0&0&0&0&0
\end{bmatrix}.
\end{eqnarray*}

Now, due to the above relation, one can deduce
\begin{eqnarray*}
\nonumber \bigg[\underline{\underline{c}}^{\hat{N}} \Pi\bigg]_{j=0}^{\hat{N}}&=&\begin{cases}
0,\hspace{+2.75cm} j < \theta \alpha b,\\
\sum\limits_{i=0}^{j-\theta \alpha b}\underline{c}_{i}\big[\Pi\big]_{i,j},\quad \quad \quad j \geq \theta \alpha b,
\end{cases}
\\&=&[\overbrace{0 \ldots  0}^{\theta \alpha b},\bold{Z}_{0}^{\theta},
\bold{Z}_{1}^{\theta},\ldots,\bold{Z}_{\hat{N}-\theta \alpha b}^{\theta}],
\end{eqnarray*}
where $\bold{Z}_{r}^{\theta},$ $r=0,1,\ldots,\hat{N}-\theta \alpha b$ are non-linear functions of the elements $\underline{c}_{0},~\underline{c}_{1},\ldots,~\underline{c}_{r}$.

Eventually we attain
\begin{equation}\label{e48}
\sum_{\theta=1}^{\infty}[\overbrace{0 \ldots  0}^{\theta \alpha b},\bold{Z}_{0}^{\theta},
\bold{Z}_{1}^{\theta},\ldots,\bold{Z}_{\hat{N}-\theta \alpha b}^{\theta}]=[\overbrace{0 \ldots  0}^{\alpha b},\bold{\tilde{Z}}_{0},
\bold{\tilde{Z}}_{1},\ldots,\bold{\tilde{Z}}_{\hat{N}- \alpha b}],
\end{equation}
in which $\bold{\tilde{Z}}_{r},$ $r=0,1,\ldots,\hat{N}-\alpha b$ are non-linear functions in terms of the components  $\underline{c}_{0},~\underline{c}_{1},\ldots,~\underline{c}_{r}$.
We substitute \eqref{e48} into \eqref{e44} and compute the unknown elements of the unknown vector $\underline{\underline{c}}$ through the recurrence relations bellow
\begin{eqnarray*}
\underline{c}_{0}&=&\frac{1}{\Gamma{(\alpha+1)}}(f_{0,0}+\lambda [\underline{K}]_{0}),\\
\vdots &~&\\
\underline{c}_{ \alpha b-1}&=&\frac{\Gamma{(\gamma (\alpha b-1)+1)}}{\Gamma{(\alpha+\gamma ( \alpha b-1)+1)}}(f_{ \alpha b-1,0}+\lambda [\underline{K}]_{ \alpha b-1} ),\\
\underline{c}_{\alpha b}&=&\frac{\Gamma{(\gamma (\alpha b)+1)}}{\Gamma{(\alpha+\gamma ( \alpha b)+1)}}
\Big(f_{ \alpha b,0}+\lambda [\underline{K}]_{ \alpha b}+\bold{\tilde{Z}}_{0}\Big),\\
\vdots &~&\\
\underline{c}_{\hat{N}}&=&\frac{\Gamma{(\gamma \hat{N} +1)}}{\Gamma{(\alpha+\gamma \hat{N}+1)}}
\Big(f_{\hat{N},0}+\lambda [\underline{K}]_{\hat{N}}+\bold{\tilde{Z}}_{\hat{N}- \alpha b}\Big).
\end{eqnarray*}

Ultimately, solving the lower-triangular system $\underline{\underline{c}}^{\hat{N}}=\underline{c}^{\hat{N}} J^{\hat{N}}$ fuels the main unknown $\underline{c}^{\hat{N}}$ whereby we can obtain the Petrov-Galerkin approximation \eqref{e31}.
\section{Error estimate}\label{sec3}
This section is dedicated to giving the convergence result of the introduced method through employing an appropriate error bound in $L^{2}$-norm.

‎Let us introduce
$\Pi_{N}^{(0,\alpha b,\gamma)}$ as the $L_{w^{(0,\alpha b,\gamma)}}^2(\Lambda)$-orthogonal projection relevant
to the fractional Jacobi space
$$\mathcal{F}_{N}^{(0,\alpha b,\gamma)}=\text{Span}\{J_k^{(0,\alpha b,\gamma)}: ~k=0,1,\ldots,N\}.$$

Meanwhile, for $p \in L_{w^{(0,\alpha b,\gamma)}}^2(\Lambda)$, we can write
\[
\bigg(p-\Pi_{N}^{(0,\alpha b,\gamma)} p,\phi\bigg)_{w^{(0,\alpha b,\gamma)}}=0,\quad \forall \phi \in \mathcal{F}_{N}^{(0,\alpha b,\gamma)}.
\]

To obtain an upper bound of truncation error $\Pi_{N}^{(0,\alpha b,\gamma)}p-p$, we first define the space
\begin{equation*}
H_{w^{(\rho,\eta)}}^m(I)=\{P : \|P\|_{m,w^{(\rho, \eta)}}< \infty,~ m \in \mathbb{N}\},
\end{equation*}
along with
‎\[‎\|P\|_{m,w^{(\rho,\eta)}}^2=\sum\limits_{l=0}^{m}{\|\partial_s^l P\|_{w^{\rho+l,\eta+l}}^2},\quad \lvert P \rvert_{m,w^{(\rho,\eta)}}=\|\partial_s^m P\|_{w^{\rho+m,\eta+m}},
‎\]
considered as the norm and semi-norm.

In this step, if we assume that the coordinate transformation $s = 2 (\frac{t}{T})^{\gamma}-1 $ associate the function $p(t)$ with $P(s)$, their derivatives will be connected as follows
\begin{eqnarray*}‎
‎D_t p:=\partial_s P(s)&=&\partial_s t‎~ ‎\partial_t p,\\‎
‎D_t^2 p:=\partial_s^2 P(s)&=&\partial_s t~\partial_t(D_t p)~,\\‎
‎\vdots‎
‎\\‎
‎D_t^n p:=\partial_s^n P(s)&=&\partial_s t~\partial_t(\partial_s t‎~ ‎\partial_t(\cdots(\partial_s t‎~ ‎\partial_t p)\cdots))‎,
‎\end{eqnarray*}‎
in which $\partial_s t=\frac{T}{2\gamma}\big(\frac{t}{T}\big)^{1-\gamma}$. Moreover, it can be deduced that
‎\begin{eqnarray*}
\|P(s)\|_{w^{(\rho,\eta)}}^2&=&\int_{I}{\lvert P(s)\rvert^2 w^{(\rho,\eta)}(s)ds}=‎
‎d^{(\rho,\eta)}\int_{\Lambda}{\lvert p(t)\rvert^2 w^{(\rho,\eta,\gamma)}(t)dt}\\
&=&‎d^{(\rho,\eta)}\|p(t)\|_{w^{(\rho,\eta,\gamma)}}^2,
\end{eqnarray*}
\begin{eqnarray*}
‎\|\partial_s^m P(s)\|_{w^{(\rho,\eta)}}^2&=&\int_{I}{\lvert \partial_s^m P(s)\rvert^2 w^{(\rho,\eta)}(s)ds}=d^{(\rho,\eta)}\int_{\Lambda}{\lvert D_t^m p(t)\rvert^2 w^{(\rho,\eta,\gamma)}(t)dt}\\
&=&d^{(\rho,\eta)}\|D_t^m p(t)\|_{w^{(\rho,\eta,\gamma)}}^2,
‎\end{eqnarray*}‎
where ‎$d^{(\rho,\eta)}=\frac{2^{\rho+\eta+1}}{T}$.

Finally, we define the transformed space‎
‎ \[H_{w^{(\rho,\eta,\gamma)}}^m(\Lambda)=\{p:~ \|p\|_{m,w^{(\rho,\eta,\gamma)}} < \infty\},
    \]
associated with the norm and semi-norm
\begin{eqnarray*}
&&\|p\|_{m,w^{(\rho,\eta,\gamma)}}^2=\sum\limits_{l=0}^{m}{d^{(\rho+l,\eta+l)}\|D_t^l p\|_{w^{(\rho+l,\eta+l,\gamma)}}^2},\\
&&\lvert p\rvert_{m,w^{(\rho,\eta,\gamma)}}=\sqrt{d^{(\rho+m,\eta+m)}}\|D_t^m p\|_{w^{(\rho+m,\eta+m,\gamma)}},
\end{eqnarray*}
and from Theorem 3.2 of \cite{am1}, the following estimation holds
\begin{equation}\label{e52}
 \|\Pi_{N}^{(0,\alpha b,\gamma)}p-p\|_{w^{(0,\alpha b,\gamma)}} \leq C N^{-m}\lvert p\rvert_{m,w^{(0,\alpha b,\gamma)}},~~m \ge 0.
\end{equation}

It is time to exhibit the convergence theorem directing the proper error bound for $y(t)-y_{N}(t)$ in $L^2$-norm.
\begin{theorem}[Convergence]\label{thm4}
‎Let $y_{N}(t)$ given by \eqref{e31} be the approximate solution of \eqref{eq1}‎. If we have
\begin{enumerate}
\item{$\int_{0}^{t}(t-s)^{{\beta}-1}g(t,s,y(s))ds \in H_{w^{(0,\alpha b,\gamma)}}^{\varepsilon}(\Lambda)$ such that
\[D^{\varepsilon+1}_{t} \left(\int_{0}^{t}(t-s)^{{\beta}-1}g(t,s,y(s))ds\right)\in C(\Lambda),~~\varepsilon \ge 0.\]}
\item{$f \in H_{w^{(0,\alpha b,\gamma)}}^{\varsigma}(\Lambda)$ such that $D^{\varsigma+1}_{t} f \in C(\Lambda),~~\varsigma \ge 0$}.
\end{enumerate}
Then the following upper bound holds for sufficiently large values of $N$
‎\begin{align*}‎
\Vert e_{N}(t)\Vert \leq C \left(\hat{N}^{-\varepsilon} |\int_{0}^{t}(t-s)^{{\beta}-1}g(t,s,y(s))ds|_{\varepsilon,w^{(0,\alpha b,\gamma)}}+\hat{N}^{-\varsigma}|f|_{\varsigma,w^{(0,\alpha b,\gamma)}}\right),
‎\end{align*}‎
‎where $e_{N}(t)=y(t)-y_{N}(t)$ dictates the error function, and $C>0$ denotes a generic constant independent of $N$.
\end{theorem}
\begin{proof}
Based on the devised numerical approach in Section \ref{sec2}, we get the following operator equation
\begin{equation*}
\Pi_{\hat{N}}^{(0,\alpha b,\gamma)} \left(‎D^{\alpha}_{C} y_{N}(t)-f(t,y_{N}(t))-\lambda \int_{0}^{t}(t-s)^{{\beta}-1}g(t,s,y_{N}(s))ds\right)=0,
\end{equation*}
and equivalently we have
\begin{equation}\label{e50}
D^{\alpha}_{C} y_{N}(t)=\Pi_{\hat{N}}^{(0,\alpha b,\gamma)} \big(f(t,y_{N}(t))\big)+\lambda \Pi_{\hat{N}}^{(0,\alpha b,\gamma)} \left(\int_{0}^{t}(t-s)^{{\beta}-1}g(t,s,y_{N}(s))ds\right),
\end{equation}
since we have  $D^{\alpha}_{C} y_{N}(t) \in \text{Span}\{J_0^{(0,\alpha b,\gamma)}, J_1^{(0,\alpha b,\gamma)},...,J_N^{(0,\alpha b,\gamma)}\}$. Subtracting \eqref{e50} from \eqref{eq1} yields
\begin{eqnarray*}
&&D^{\alpha}_{C}e_{N}(t)=f(t,y(t))-\Pi_{\hat{N}}^{(0,\alpha b,\gamma)} \big(f(t,y_{N}(t))\big)\\
&&+\lambda \Bigg(\int_{0}^{t}(t-s)^{{\beta}-1}g(t,s,y(s))ds-\Pi_{\hat{N}}^{(0,\alpha b,\gamma)} \left(\int_{0}^{t}(t-s)^{{\beta}-1}g(t,s,y_{N}(s))ds\right)\Bigg),
\end{eqnarray*}
which can be restated as
\begin{multline}\label{eqq4}
 D^{\alpha}_{C}e_{N}(t)=(f-\bar{f})+\lambda \int_{0}^{t}(t-s)^{{\beta}-1}\big(g(t,s,y(s))-g(t,s,y_{N}(s))\big)ds\\
+e_{\Pi_{\hat{N}}^{(0,\alpha b,\gamma)}}\left(\bar{f}+\lambda \int_{0}^{t}(t-s)^{{\beta}-1}g(t,s,y_{N}(s))ds\right),
\end{multline}
where $e_{\Pi_{\hat{N}}^{(0,\alpha b,\gamma)}}(z)=z-\Pi_{\hat{N}}^{(0,\alpha b,\gamma)}(z)$ and $\bar{f}=f(t,y_{N}(t))$. We enforce the Riemann-Liouville fractional integral operator of order $\alpha$ on \eqref{eqq4} and utilize the relation $I^{\alpha}D_{C}^{\alpha}e_{N}(t)=e_{N}(t)$ whereby we deduce
\begin{eqnarray*}
&&e_{N}(t)=I^{\alpha}(f-\bar{f})\\
&&+\lambda \int_{0}^{t}(t-s)^{{\alpha}-1}\left(\int_{0}^{s}(s-\tau)^{{\beta}-1}\big(g(s,\tau,y(\tau))-g(s,\tau,y_{N}(\tau))\big)d\tau\right)ds+R,
\end{eqnarray*}
along with
\[
R=I^\alpha e_{\Pi_{\hat{N}}^{(0,\alpha b,\gamma)}}\left(\bar{f}+\lambda \int_{0}^{t}(t-s)^{{\beta}-1}g(t,s,y_{N}(s))ds\right).
\]

Due to the Lipschitz assumption on f and g, we arrive at
\begin{align}\label{e53}
\nonumber \lvert e_{N}(t)\rvert &\leq L_{1}I^{\alpha}\lvert e_{N}(t)\rvert+\frac{\lambda L_{2}}{\Gamma(\alpha)}\int_{0}^{t}(t-s)^{{\alpha}-1}\left(\int_{0}^{s}(s-\tau)^{{\beta}-1}\lvert e_{N}(\tau)\rvert d\tau\right)ds+\lvert R \rvert\\
&=L_{1}I^{\alpha}\lvert e_{N}(t) \rvert+\lambda L_{2}\Gamma(\beta) I^{\alpha+\beta}\lvert e_{N}(t)\rvert +\lvert R \rvert.
\end{align}
In addition to this, we can write
\begin{multline}\label{e54}
L_{1}I^{\alpha}\lvert e_{N}\rvert+\lambda L_{2}\Gamma(\beta) I^{\alpha+\beta}\lvert e_{N}\rvert \leq \max(L_{1},\lambda L_{2}\Gamma(\beta))\left(I^{\alpha}\lvert e_{N} \rvert+ I^{\alpha+\beta}\lvert e_{N}\rvert\right)\\
 =\max(L_{1},\lambda L_{2}\Gamma(\beta))\Bigg(\int_{0}^{t}(t-s)^{{\alpha}-1}\left(\frac{1}{\Gamma(\alpha)}+\frac{1}{\Gamma(\alpha+\beta)}(t-s)^{\beta}\right)\lvert e_{N}(s)\rvert ds\Bigg)\\
 \leq \aleph \int_{0}^{t}(t-s)^{{\alpha}-1}\lvert e_{N}(s)\rvert ds,\hspace{+6.5cm}
\end{multline}
in which
\[
\aleph=\max(L_{1},\lambda L_{2}\Gamma(\beta))\left(\frac{1}{\Gamma(\alpha)}+\frac{T^{\beta}}{\Gamma(\alpha+\beta)}\right).
\]
Substituting \eqref{e54} into \eqref{e53} and employing Gronwall's inequality, i.e., Lemma 6 of \cite{rrrh35}, it can be concluded that
‎\begin{align*}‎
‎\|e_{N}(t)\| \leq C \left\|R\right\|‎.
‎\end{align*}‎
We rewrite the above inequality in the following sense
‎\begin{equation}\label{e390}
‎\|e_{N}(t)\|  \leq C
‎\Vert e_{\Pi_{\hat{N}}^{(0,\alpha b,\gamma)}}\left(\bar{f}+\lambda \int_{0}^{t}(t-s)^{{\beta}-1}g(t,s,y_{N}(s))ds\right) \Vert_{w^{(0,\alpha b,\gamma)}},
‎\end{equation}‎
due to the Cauchy-Schwarz inequality and some manipulations.

Currently, the suitable upper bounds are sought for each term of the right-hand side of \eqref{e390}. In this regard, utilizing the estimation \eqref{e52} and the first-order Taylor formula give
\begin{eqnarray}\label{e55}
\nonumber &&\|e_{\Pi_{\hat{N}}^{(0,\alpha b,\gamma)}}\left(\lambda \int_{0}^{t}(t-s)^{{\beta}-1}g(t,s,y_{N}(s))ds\right)\|_{w^{(0,\alpha b,\gamma)}} \\
\nonumber&&\leq C \hat{N}^{-\varepsilon} \lvert \int_{0}^{t}(t-s)^{{\beta}-1}g(t,s,y_{N}(s))ds\rvert_{\varepsilon,w^{(0,\alpha b,\gamma)}}\\
\nonumber&&\leq  C \hat{N}^{-\varepsilon} \Bigg(\lvert \int_{0}^{t}(t-s)^{{\beta}-1}g(t,s,y(s))ds\rvert_{\varepsilon,w^{(0,\alpha b,\gamma)}}+C_{1} \|g(t,s,y(s))-g(t,s,y_{N}(s)) \|\Bigg) \\
&&\leq C \hat{N}^{-\varepsilon} \Bigg(\lvert \int_{0}^{t}(t-s)^{{\beta}-1}g(t,s,y(s))ds\rvert_{\varepsilon,w^{(0,\alpha b,\gamma)}}+C_{1} L_{2} \| e_{N}\|\Bigg),
\end{eqnarray}
due to the Lipschitz assumption on g. Here, $C_{1}>0$ is a generic constant independent of $N$. Proceeding the same way as \eqref{e55}, we derive
\begin{equation}\label{e56}
\|e_{\Pi_{\hat{N}}^{(0,\alpha b,\gamma)}}(\bar{f})\|_{w^{(0,\alpha b,\gamma)}} \leq C \hat{N}^{-\varsigma} \lvert \bar{f}\rvert_{\varsigma,w^{(0,\alpha b,\gamma)}} \leq C \hat{N}^{-\varsigma} \Big(\lvert f\rvert_{\varsigma,w^{(0,\alpha b,\gamma)}}+C_{2}L_{2}\| e_{N}\|\Big),
\end{equation}
where $C_{2}>0$ denotes a generic constant independent of $N$.

Inserting \eqref{e55} and \eqref{e56} into \eqref{e390} concludes
\begin{multline*}
\Vert e_{N}(t)\Vert -C( \hat{N}^{-\varepsilon}C_{1} L_{2}+\hat{N}^{-\varsigma}C_{2}L_{2}) \| e_{N}(t)\| \\
 \leq C \left(\hat{N}^{-\varepsilon} \lvert\int_{0}^{t}(t-s)^{{\beta}-1}g(t,s,y(s))ds\rvert_{\varepsilon,w^{(0,\alpha b,\gamma)}}+\hat{N}^{-\varsigma}\lvert f\rvert_{\varsigma,w^{(0,\alpha b,\gamma)}}\right).
\end{multline*}

Evidently, the desired result can be attained for sufficiently large values of $N$.
\end{proof}
\section{Numerical illustration}\label{sec4}
This section is devoted to confirming the effectiveness and productivity of the proposed implementation through demonstrating the numerical experiments derived from ‎solving some non-linear weakly singular FIEDs‎‎. In this respect, this section is organized in the following sense
\begin{itemize}
  \item[$\bullet$] To assess the computational capability of the introduced strategy, we illustrate some essential properties including numerical errors $\bold{e}(N)=\|e_{N}(t)\|$ and CPU-time elapsed.

\item[$\bullet$] The stability of the method is also examined via approximating the highly oscillatory solution of a test problem associated with the long domain $\Lambda$ and large values of $N$.

\item[$\bullet$] The predominance of the suggested approach is assessed by comparing our results to those obtained by a modification of hat functions (MHFs) introduced in \cite{18}.
\end{itemize}
The computation is conducted by means of Mathematica v11.2‎, ‎running in a computer system with an Intel (R) Core (TM) i5-4210U CPU @ 2.40 GHz‎‎.
‎\begin{example} \label{exm1}‎
Let us consider the non-linear weakly singular FIDEs
‎\begin{align*}
‎‎\begin{cases}‎
‎D^{\frac{3}{2}}_C y(t)=f(t)+ \int_{0}^{t}(t-s)^{-\frac{3}{4}}g(t,s,y(s))ds,\quad ‎t \in [0,1],\\‎
y(0)=0.
‎\end{cases}‎
‎\end{align*}‎
The source function $f(t)$ is chosen in a way that the exact solution is
\begin{equation*}
y(t)=E_{\frac{3}{2}}(t^{\frac{3}{2}})-1,
\end{equation*}
and
\begin{align*}‎
g(t,s,y(s))=\frac{1}{2}sJ_{0}(t^{\frac{7}{4}})\sin{(y(s))}+s^{\frac{5}{2}}y^{4}(s)+t^{\frac{1}{2}}s^{\frac{1}{4}}.
‎\end{align*}‎
\end{example}

Trivially, we have $y(t)=O(t^{\frac{3}{2}})$ which coincide with the result of Theorem \ref{th202}. $E_{c}(t)$ denotes Mittag-Leffler function‎, and for integer number $d$, $J_{d}(t)$ is known as Bessel function‎.

 We assess this problem by means of the proposed implementation and report the results in Table \ref{tab1} and Fig. \ref{fig1000}. From Table \ref{tab1}, it is reasoned that the approximate solutions are highly accurate, it is because the numerical errors are declined regularly in the short CPU-time used particularly for the large degrees of approximation $N$. In addition, the semi-log depiction of the numerical errors demonstrated in Fig. \ref{fig1000} confirms the well-known exponential accuracy predicting in Theorem \ref{thm4} caused by the linear variations of the semi-log depiction of errors versus $N$ (notice that we have
$\varepsilon,~\varsigma=\infty$ in Theorem \ref{thm4}).
\begin{table}[htbp]
{\footnotesize
\caption{The numerical consequences of Example \ref{exm1} for different degrees of approximation $N$.}\label{tab1}
\begin{center}
\setlength\tabcolsep{4pt}
\begin{tabular}{@{}c|cc@{}}
\hline
\hline

 $N$ &  $\bold{e}(N)$ & $\textit{CPU-time (sec)}$ \\
 \hline
 \hline
8 &6.93E-2&1.51\\
16 & 6.44E-3& 18.92 \\
32 &6.51E-7 &113.13 \\
64 & 2.04E-15& 555.19\\
\hline
\hline
\end{tabular}\end{center}}
\end{table}
\begin{figure}[htbp]
 \centering
\includegraphics[width=6cm]{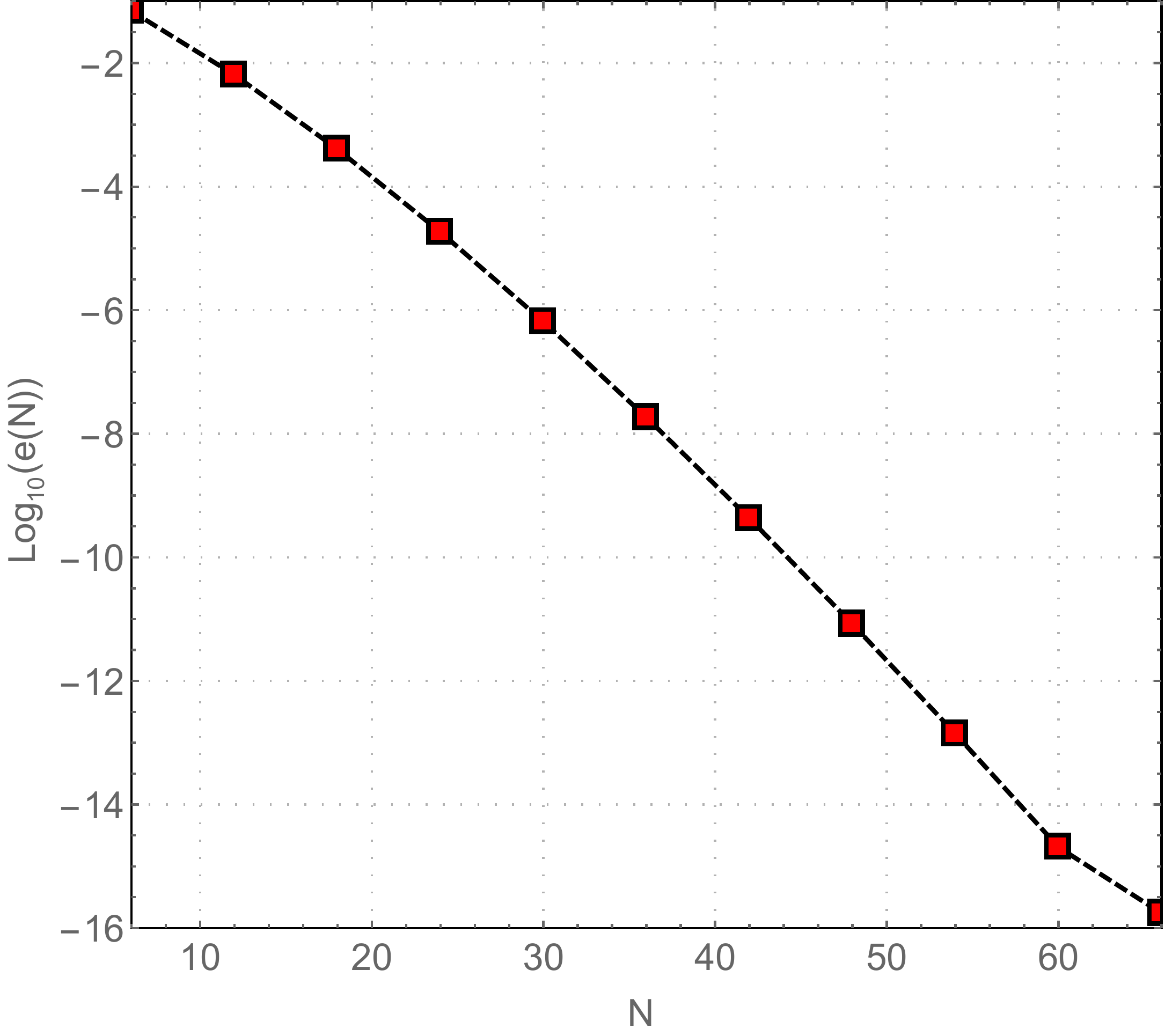}
 \caption{Semi-log depiction of the numerical errors versus $N$ for Example \ref{exm1}.}
\label{fig1000}‎
\end{figure}
‎\begin{example}\label{exm4}‎
Let us give the following highly oscillatory non-linear weakly singular FIDEs
\begin{equation*}
‎‎\begin{cases}‎
‎D^{\frac{1}{2}}_C y(t)=f(t)+ \int_{0}^{t}(t-s)^{-\frac{1}{2}}g(t,s,y(s))ds,\quad ‎t \in [0,2 \Pi],\\‎
y(0)=0.
‎\end{cases}‎
\end{equation*}
The forcing function $f(t)$ is selected in a way that
\begin{equation*}
g(t,s,y(s))=t^{2}s^{\frac{3}{2}}y^{2}(s)+5 {_{2}F_{2}}(\{\frac{1}{4},\frac{3}{4}\};\{\frac{1}{5},\frac{2}{5}\};‎ -‎\frac{t^{\frac{1}{2}}}{2}),
\end{equation*}
and $y(t)=t \sin 100 t^{\frac{1}{2}}$.
\end{example}

${_{\xi}F_{\xi}}(\{a_{1},\ldots,a_{\xi}\};\{b_{1},\ldots,b_{\xi}\};t)$ denotes the generalized hypergeometric function‎. We apply the introduced scheme to this problem and the derived consequences are illustrated in Fig. \ref{fg4} and Fig. \ref{fg2500}. Needless to mention, the highly oscillatory behavior of the solution may cause instability in approximation, particularly for large degrees of approximation $N$. Regardless of this fact, the numerical results demonstrate that our scheme, however, produces highly accurate approximate solutions. Indeed, From Fig. \ref{fg4}, it can be concluded that the method is on the path of convergence for $N>680$, and the effective computational performance of our strategy let the numerical errors decline regularly specifically for large degrees of approximation $N$. Furthermore, the well-known exponential accuracy is confirmed caused by the linear variations of semi-log depiction of errors versus $N$.

‎
\begin{figure}[!tbp]
  \centering
  \begin{minipage}[b]{0.4\textwidth}
    \includegraphics[width=.9\textwidth]{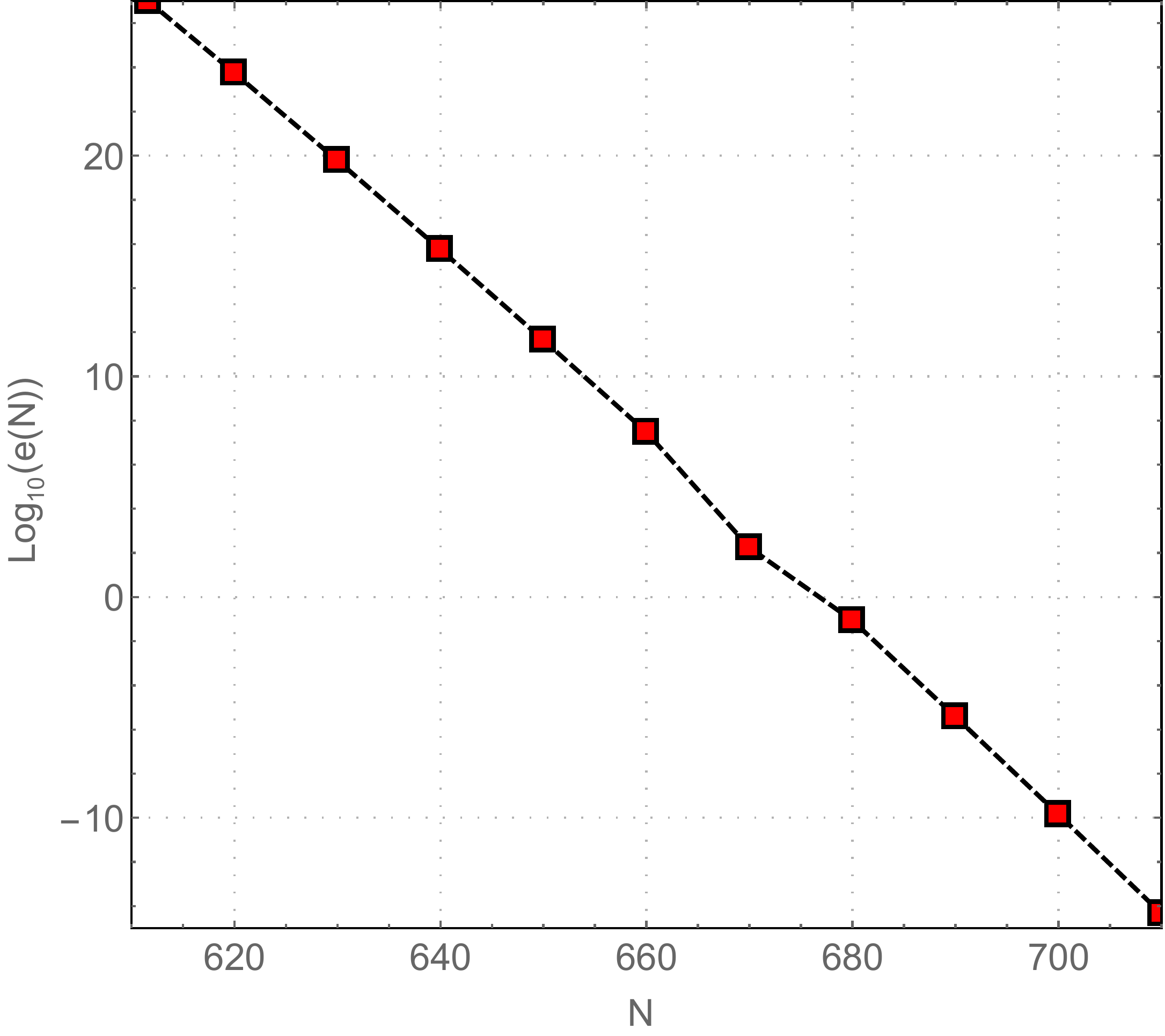}
    \caption{Semi-log depiction of the numerical errors versus $N$ for Example \ref{exm4}.}\label{fg4}
  \end{minipage}
  \hfill
  \begin{minipage}[b]{0.4\textwidth}
    \includegraphics[width=1.3\textwidth]{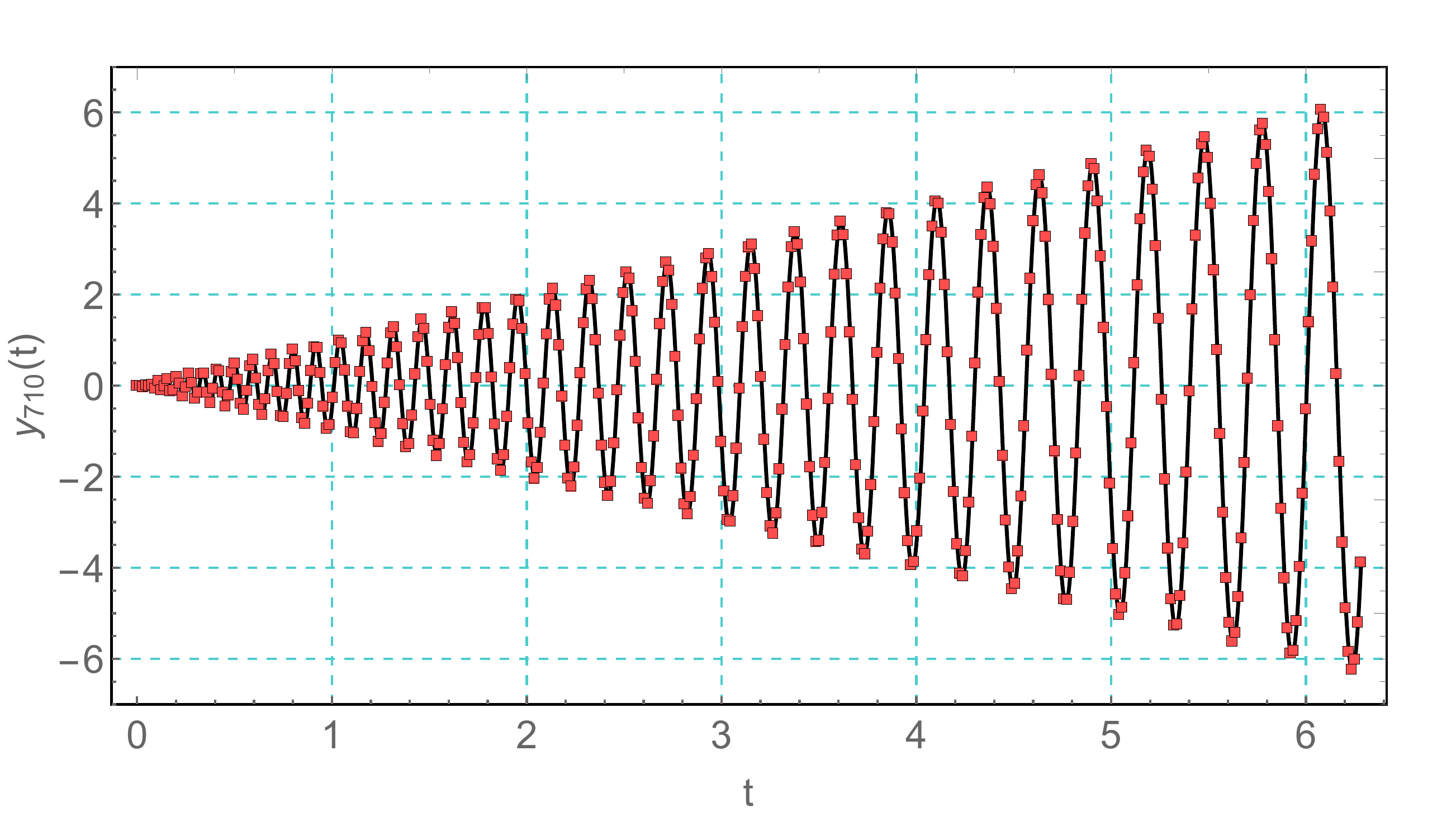}
    \caption{Graphs of the exact solution (solid line) and the approximate solution (red squares) of Example \ref{exm4} for $N=710$.}\label{fg2500}
  \end{minipage}
  \hspace{+1.2cm}
\end{figure}

\begin{example}\label{exm2}
\cite{18} Consider the following non-linear weakly singular FIDEs
‎\begin{align}\label{eq45}‎
‎‎\begin{cases}‎
‎D^{\frac{2}{3}}_C y(t)=f(t,y(t))+ \int_{0}^{t}(t-s)^{-\frac{1}{2}}y^{2}(s)ds,\quad ‎t \in [0,1],\\‎
y(0)=0,
‎\end{cases}‎
‎\end{align}‎
where
\begin{align*}
f(t,y(t))=\frac{3 \Gamma(\frac{1}{2})}{4 \Gamma(\frac{11}{6})}t^{\frac{5}{6}}-t^{\frac{5}{2}}-\frac{32}{35}t^{\frac{7}{2}}+t y(t).
\end{align*}\end{example}

The exact solution of this problem is $y(t)=t^{\frac{3}{2}}$. Considering $\gamma=\frac{1}{6}$, this problem is solved via the implemented scheme, and the exact solution is obtained with the degree of approximation $9$, in machine precision. Also, in \cite{18}, the approximate solution of \eqref{eq45} is computed by means of a modification of hat functions (MHFs), and the absolute errors at some selected grid points derived in \cite{18} for various values of $n$ are listed in Table \ref{tab2}. In this method, $n$ directs the number of uniform sub-intervals. We refer the reader to \cite{18} for more details about this method. Comparison results justify the dominance of our suggested method over the presented scheme in \cite{18}.
\begin{table}[htbp]
{\footnotesize
  \caption{The absolute errors at some selected grid points of Example \ref{exm2} with different values of $n$ utilizing the suggested method in \cite{18}.}  \label{tab2}
\begin{center}
\begin{tabular}{c|c}
\hline
\hline

 $n$ &  $Error$  \\
 \hline
 \hline
2 &7.09E-2   \\
4 &1.40E-2 \\
8 &3.31E-3 \\
16 &8.85E-4   \\
32 &2.46E-4 \\
64 &6.92E-5 \\
128 &2.42E-5 \\
256 &8.57E-6  \\
512 &3.03E-6   \\
1024 &1.07E-6 \\
\hline
\hline
\end{tabular}
\end{center}
}
\end{table}
\section{Conclusions}\label{sec5}

A comprehensive survey of the existence, uniqueness, and smoothness properties of the solution of \eqref{eq1} was presented, and in particular, was demonstrated that the solution has a singularity at the origin. Taking into account the smoothness of the solution we proposed a creative strategy based on the spectral Petrov-Galerkin method to solve \eqref{eq1} numerically. This strategy offered some recurrence relations for deriving the approximate solution rather than solving a non-linear complex algebraic system. Finally, our implementation drove us to verify the spectral accuracy of the proposed method through the convergence theorem and approximating some illustrative examples. This strategy enables us to attack a vast majority of non-linear fractional functional equations, which would possibly motivate us to do research on it in the future.

\section*{Acknowledgements}

This work is funded by national funds through the FCT - Funda\c c\~ao para a Ci\^encia e a Tecnologia, I.P., under the scope of the projects UIDB/00297/2020 and UIDP/00297/2020 (Center for Mathematics and Applications).

\end{document}